\documentclass[11pt,a4paper]{article}

\usepackage{amsmath}
\usepackage{amsthm}
\usepackage{amssymb}
\usepackage{cain_arxiv}
\usepackage{hyperref}
\usepackage{longtable}
\usepackage{booktabs}
\usepackage{pdflscape}

\usepackage{tikz}
\usetikzlibrary{calc,positioning}

\theoremstyle{definition}
\newtheorem{definition}{Definition}[section]
\newtheorem{example}[definition]{Example}
\newtheorem{question}[definition]{Question}

\theoremstyle{plain}
\newtheorem{proposition}[definition]{Proposition}

\newtheorem{theorem}[definition]{Theorem}

\numberwithin{equation}{section}

\def\fullref#1#2{%
  \ifdefined\hyperref%
    {\hyperref[#2]{#1\space\penalty 200\relax\ref*{#2}}}%
  \else%
    {#1\space\penalty 200\relax\ref{#2}}%
  \fi%
}

\newcommand{\defterm}[1]{\textit{#1}}

\newcommand{\cgen}[1]{#1^{\#}}
\newcommand{\malcgen}[1]{#1^{{\rm M}}}
\newcommand{\ccgen}[1]{#1^{{\rm C}}}
\newcommand{\lccgen}[1]{#1^{{\rm LC}}}
\newcommand{\rccgen}[1]{#1^{{\rm RC}}}

\newcommand{\pres}[3][]{#1\langle #2\:#1|\:#3 #1\rangle}
\newcommand{\sgpres}[3][]{\mathrm{Sg}#1\langle #2\:#1|\:#3 #1\rangle}
\newcommand{\sgmpres}[3][]{\mathrm{SgM}#1\langle #2\:#1|\:#3 #1\rangle}
\newcommand{\sgcpres}[3][]{\mathrm{SgC}#1\langle #2\:#1|\:#3 #1\rangle}
\newcommand{\sglcpres}[3][]{\mathrm{SgLC}#1\langle #2\:#1|\:#3 #1\rangle}
\newcommand{\sgrcpres}[3][]{\mathrm{SgRC}#1\langle #2\:#1|\:#3 #1\rangle}
\newcommand{\invpres}[3][]{\mathrm{Inv}#1\langle #2\:#1|\:#3 #1\rangle}

\def\mall#1{#1^{\mathsf{L}}}
\def\malr#1{#1^{\mathsf{R}}}

\newcommand{\imreduces}{\rightarrow}

\def\adjzero#1{{#1}^{{\tt 0}}}

\newcommand{\nset}{\mathbb{N}}
\newcommand{\zset}{\mathbb{Z}}
\newcommand{\rset}{\mathbb{R}}

\DeclareMathOperator{\im}{im}

\newcommand{\emptyword}{\varepsilon}
\newcommand{\rel}[1]{\mathcal{#1}}
\newcommand{\elt}[1]{\overline{#1}}
\newcommand{\rep}[1]{\underline{#1}}

\DeclareMathOperator{\gR}{\mathcal{R}}
\DeclareMathOperator{\gL}{\mathcal{L}}
\DeclareMathOperator{\gD}{\mathcal{D}}
\DeclareMathOperator{\gJ}{\mathcal{J}}

\newcommand{\cf}{\mathrm{cf}}
\newcommand{\scf}{\mathrm{scf}}

\newcommand{\rev}{\mathrm{rev}}

\newcommand{\fgt}{{\sc FGT}}

\newcommand{\fdt}{\textsc{FDT}}
\newcommand{\FP}{\mathrm{FP}}

\newcommand{\bbe}{\mathbb{E}}
\newcommand{\bbp}{\mathbb{P}}
\newcommand{\bbq}{\mathbb{Q}}
\newcommand{\bbr}{\mathbb{R}}
\newcommand{\bbs}{\mathbb{S}}

\begin{document}

\title{For a few elements more: A~survey of finite Rees index}
\author{Alan J. Cain and Victor Maltcev}
\date{}
\thanks{During the writing of this paper, the first
  author was supported by the European
  Regional Development Fund through the programme {\sc COMPETE} and by
  the Portuguese Government through the {\sc FCT} (Funda\c{c}\~{a}o
  para a Ci\^{e}ncia e a Tecnologia) under the project
  {\sc PEst-C}/{\sc MAT}/{\sc UI}0144/2011 and through an {\sc FCT} Ci\^{e}ncia 2008
  fellowship.}

\maketitle

\address[AJC]{%
Centro de Matem\'{a}tica, Universidade do Porto, \\
Rua do Campo Alegre 687, 4169--007 Porto, Portugal
}
\email{%
ajcain@fc.up.pt
}
\webpage{%
www.fc.up.pt/pessoas/ajcain/
}

\address[VM]{%
Mathematical Institute, University of St Andrews,\\
North Haugh, St Andrews, Fife KY16 9SS, United Kingdom
}
\email{%
victor.maltcev@gmail.com
}

\begin{abstract}
This paper makes a comprehensive survey of results relating to finite
Rees index for semigroups. In particular, we survey of the state of
knowledge on whether various finiteness properties (such as finite
generation, finite presentability, automaticity, and hopficity) are
inherited by finite Rees index subsemigroups and extensions. We survey
other properties that are invariant under passing to finite Rees index
subsemigroups and extensions, such as the cofinality and number of
ends. We prove some new results: inheritance of word-hyperbolicity by
finite Rees index subsemigroups, and inheritance of (geometric)
hyperbolicity by finite Rees index extensions and subsemigroups within
the class of monoids of finite geometric type. We also give some
improved counterexamples. All the results are summarized in a table.
\end{abstract}


\section{Introduction}

One of the most important ideas in Group Theory is the notion of
\emph{index}.  It appears in many important theorems: for instance,
Gromov's Growth Theorem, the Muller--Schupp Theorem, and a whole
strand of results of the Reidemeister--Schreier-type, that is, on
preservation of various conditions to subgroups or extensions of finite
index.

\begin{landscape}
\begin{longtable}{lc@{\hspace{8mm}}c@{\space}lc@{\space}l}
\caption[Summary of properties inherited by small extensions or large subsemigroups]{\pdfbookmark[1]{Summary table}{pdf:summarytable}Summary of properties inherited by small extensions or large subsemigroups} \label{tbl:summary} \\
\toprule 
& & \multicolumn{4}{c}{{\it Inherited by}} \\
\cmidrule(r){3-6} 
{\it Property} & {\it See} & \multicolumn{2}{l}{{\it Small extensions}} & \multicolumn{2}{l}{{\it Large subsemigroups}} \\
\midrule \endfirsthead
\caption{Summary of properties inherited by small extensions or large subsemigroups (continued)} \\
\toprule 
& & \multicolumn{4}{c}{{\it Inherited by}} \\
\cmidrule(r){3-6} 
{\it Property} & {\it See} & \multicolumn{2}{l}{{\it Small extensions}} & \multicolumn{2}{l}{{\it Large subsemigroups}} \\
\midrule \endhead
\bottomrule \endfoot
Generators and relations \\
\quad Finite generation & \fullref{\S}{sec:fingen} & Y & (Trivial) & Y & \cite{jura_ideals} \\
\quad Finite presentation & \fullref{\S}{sec:finpres} & Y & \cite[Th.~4.1]{ruskuc_largesubsemigroups} & Y & \cite[Th.~1.3]{ruskuc_largesubsemigroups} \\
\quad Finite left-/right-cancellative presentation & \fullref{\S}{sec:cancpres} & Y & \cite[Th.~3]{crr_finind} & Y & \cite[Th.~3]{crr_finind} \\
\quad Finite cancellative presentation  & \fullref{\S}{sec:cancpres} & Y & \cite[Th.~2]{crr_finind} & Y & \cite[Th.~2]{crr_finind} \\
\quad Finite Malcev presentation & \fullref{\S}{sec:cancpres} & Y & \cite[Th.~1]{crr_finind} & Y & \cite[Th.~1]{crr_finind} \\
\quad Finite inverse semigroup presentation & \fullref{\S}{sec:invpres} & ? &  & ? &  \\
\quad Soluble word problem & \fullref{\S}{sec:wordproblem} & Y & \cite[Th.~5.1(i)]{ruskuc_largesubsemigroups} & Y & (Trivial) \\ 
\quad Finite complete rewriting system & \fullref{\S}{sec:fcrs} & Y & \cite[Th.1]{wang_fcrsfdt} & Y & \cite[Th.~1.1]{wong_fcrs} \\
Homological properties \\
\quad Finite derivation type & \fullref{\S}{sec:fdt} & Y & \cite[Th.1]{wang_fcrsfdt} & ? & \\
\quad $\FP_n$ & \fullref{\S}{sec:fpn} & ? & & N & \cite[\S~8]{gray_homological} \\
\quad Finite left-cohomological dimension & \fullref{\S}{sec:fcd} & ? & & N & \cite[\S~8]{gray_homological} \\
\quad Finite left- \& right-cohomological dimension & \fullref{\S}{sec:fcd} & ? & & N & \cite[\S~8]{gray_homological} \\
Residual finiteness & \fullref{\S}{sec:residualfin} & Y & \cite[Co.~4.6]{ruskuc_syntactic} & Y & (Trivial) \\
Periodicity and relation properties \\
\quad Local finiteness & \fullref{\S}{sec:localfin}  & Y & \cite[Th.~5.1(ii)]{ruskuc_largesubsemigroups} & Y & (Trivial) \\
\quad Local finite presentation & \fullref{\S}{sec:localfin}  & Y & \cite[Th.~5.1(iii)]{ruskuc_largesubsemigroups} & Y & (Trivial) \\
\quad Periodicity & \fullref{\S}{sec:localfin}  & Y & \cite[Th.~5.1(iv)]{ruskuc_largesubsemigroups} & Y & (Trivial) \\
\quad Global torsion & \fullref{\S}{sec:globaltorsion}  & Y & \cite[Th.~8.1]{gray_ideals} & N & \cite[Re~8.3]{gray_ideals} \\
\quad Eventual regularity & \fullref{\S}{sec:eventualreg}  & Y & \cite[Th.~9.2]{gray_ideals} & Y & \cite[Th.~9.2]{gray_ideals} \\
Hopficity \& co-hopficity \\*
\quad Hopficity & \fullref{\S}{sec:hopf} & N & \cite[\S~2]{maltcev_hopfian} & N & \cite[\S~2]{maltcev_hopfian} \\
\quad Hopficity \& finite generation & \fullref{\S}{sec:hopf} & Y & \cite[Main Th.]{maltcev_hopfian} & N & \cite[\S~5]{maltcev_hopfian} \\
\quad Co-hopficity & \fullref{\S}{sec:hopf} & N & \cite[Ex.~4.6]{cm_hopf} & N & \cite[Ex.~4.1]{cm_hopf} \\
\quad Co-hopficity \& finite generation & \fullref{\S}{sec:hopf} & Y & \cite[Th.~4.2]{cm_hopf} & N & \cite[Ex.~4.1]{cm_hopf} \\
Subsemigroups \\
\quad Finitely many subsemigroups & \fullref{\S}{sec:ideals} & Y & \cite[\S~11]{ruskuc_largesubsemigroups} & Y & \cite[\S~11]{ruskuc_largesubsemigroups} \\
\quad All subsemigroups are large & \fullref{\S}{sec:ideals} & Y & \cite[\S~11]{ruskuc_largesubsemigroups} & Y & \cite[\S~11]{ruskuc_largesubsemigroups} \\
\quad Minimal subsemigroup & \fullref{\S}{sec:ideals} & Y & \cite[\S~11]{ruskuc_largesubsemigroups} & N & \cite[\S~11]{ruskuc_largesubsemigroups} \\
\multicolumn{6}{l}{Ideals, Green's relations, and related properties}\\
\quad Finitely many left ideals & \fullref{\S}{sec:ideals} & Y & \cite[Th.~10.4]{ruskuc_largesubsemigroups} & Y & \cite[Th.~
10.4]{ruskuc_largesubsemigroups} \\
\quad All left ideals are large & \fullref{\S}{sec:ideals} & N & \cite[Th.~10.5]{ruskuc_largesubsemigroups} & Y & \cite[Th.~10.5]{ruskuc_largesubsemigroups} \\
\quad Minimal left ideal & \fullref{\S}{sec:ideals} & Y & \cite[Th.~10.3]{ruskuc_largesubsemigroups} & N & \cite[Th.~10.3]{ruskuc_largesubsemigroups} \\
\quad Finitely many ideals & \fullref{\S}{sec:ideals} & Y & \cite[\S~11]{ruskuc_largesubsemigroups} & Y & \cite[Th.~5.1]{gray_ideals} \\
\quad All ideals are large & \fullref{\S}{sec:ideals} & N & \cite[Th.~10.5]{ruskuc_largesubsemigroups} & ? & \\
\quad Minimal ideal & \fullref{\S}{sec:ideals} & Y & \cite[\S~11]{ruskuc_largesubsemigroups} & N & \cite[\S~11]{ruskuc_largesubsemigroups} \\
\quad $\gJ = \gD$ & \fullref{\S}{sec:ideals} & Y & \cite[Th.~4.1]{gray_ideals} & N & \cite[Ex.~4.6]{gray_ideals} \\
\quad Stability & \fullref{\S}{sec:ideals} & Y & \cite[Th.~3.2]{gray_ideals} & Y & \cite[Th.~3.2]{gray_ideals} \\
\quad $\min_{\gR}$ & \fullref{\S}{sec:ideals} & Y & \cite[Th.~6.1]{gray_ideals} & Y & \cite[Th.~6.1]{gray_ideals} \\
\quad $\min_{\gJ}$ & \fullref{\S}{sec:ideals} & Y & \cite[Th.~6.4]{gray_ideals} & Y & \cite[Th.~6.4]{gray_ideals} \\
Automata \\*
\quad Automaticity & \fullref{\S}{sec:auto} & Y & \cite[Th.~1.1]{hoffmann_autofinrees} & Y & \cite[Th.~1.1]{hoffmann_autofinrees} \\*
\quad Asynchronous automaticity & \fullref{\S}{sec:auto} & ? & & Y & \cite[Th.~10.2]{cgr_greenindex} \\
\quad Word-hyperbolicity & \fullref{\S}{sec:wordhyp} & ? & & Y & \fullref{Theorem}{thm:wordhyplarge} \\
\quad Markovicity & \fullref{\S}{sec:markov} & Y & \cite[Th.~16.1]{cm_markov} & Y & \cite[Th.~16.1]{cm_markov} \\
\quad Automatic presentation & \fullref{\S}{sec:fap} & N & \cite[Pr.~6.3]{cort_const} & Y & \cite[Pr.~6.1]{cort_const} \\
\quad Unary automatic presentation & \fullref{\S}{sec:fap} & N & \cite[Ex.~33]{crt_unaryfa} & Y & \cite[Pr.~31]{crt_unaryfa} \\
Geometric properties \\
\quad Hyperbolicity & \fullref{\S}{sec:hyperbolic} & ? & & N & (Trivial) \\
\quad Hyperbolicity \& finite geometric type & \fullref{\S}{sec:hyperbolic} & Y & \fullref{Theorem}{thm:hyperbolicfgt} & Y & \fullref{Theorem}{thm:hyperbolicfgt} \\
Bergman's property & \fullref{\S}{sec:bergman} & Y & \cite[Th.~3.2]{maltcev_bergman} & ? & \\
\end{longtable}
\end{landscape}

If one wants to prove analogues of such results in semigroup theory,
some notion of the index of a subsemigroup is required. There are
several definitions, each with its own advantages, which we will
briefly survey (see~\fullref{\S~}{sec:alternative}). But in respect of
obtaining Reidemeister--Schreier-type results, practice has shown
that the most successful definition is the following:

\begin{definition}
The \defterm{Rees index} of a subsemigroup $T$ of a semigroup $S$ is
defined as $|S-T|+1$. In this case $T$ is a \defterm{large
  subsemigroup} of $S$, and $S$ is a \defterm{small extension} of $T$.
\end{definition}

The definition was introduced by Jura~\cite{jura_ideals}, and in the
case where $T$ is an ideal, the Rees index of $T$ is $S$ is the
cardinality of the factor semigroup $S/T$.

The most interesting feature of Reidemeister--Schreier type results
for Rees index are the rewriting techniques often involved in proving
or disproving the inheritance of a given finiteness condition by large
subsemigroups and small extensions.

In this paper we aim to survey the known results on Rees index. This
will be not only a comprehensive description of quite a large number
of Reidemeister--Schreier type results, but will also show how finite
Rees index interacts with geometric conditions on semigroups, and how
it preserves certain other (non-finiteness) semigroup properties. We
will also prove some new results and provide new counterexamples, some
of which sharpen previously-known results.

For reader's convenience, we summarize all results in
\fullref{Table}{tbl:summary}.

\section{Alternative notions of index}
\label{sec:alternative}

A major weakness in the notion of Rees index is that it does not
generalize the group index. Indeed, even the notion of \emph{finite}
Rees index does not generalize finite group index: if $G$ is an
infinite group and $H$ a proper subgroup, then the Rees index of $H$
in $G$ is always infinite. To address this, three other notions of
semigroup index have been suggested as alternatives to the Rees index.

The earliest was introduced by Grigorchuk~\cite{grigorchuk_semigroups}:

\begin{definition}
Let $S$ be a semigroup and $T$ a subsemigroup of $S$. The subsemigroup
$T$ of $S$ has finite \defterm{Grigorchuk index} if there exists a
finite subset $F$ such that for every $s\in S$ there exists $f\in F$
with $sf\in T$.
\end{definition}

This notion of index works perfectly in generalizing the celebrated
Gromov Growth Theorem to the class of finitely generated can\-cel\-la\-tive
semigroups, but fails to preserve such basic finiteness conditions as
finite generation.

The next alternative is the (right) syntactic index, introduced by
Ru\v{s}kuc \& Thomas~\cite{ruskuc_syntactic}:

\begin{definition}
Let $S$ be a semigroup and $T$ a subsemigroup of $S$. Let $\sigma$ be
the relation $(T \times T) \cup ((S-T) \times (S-T))$. Let
$\sigma_{\rm R}$ and $\sigma_{\rm L}$ be, respectively, the largest
right congruence and largest left congruence contained in
$\sigma$. The \defterm{right syntactic index} of $T$ in $S$, denoted
$[S:T]_{\rm R}$, is the number of $\sigma_{\rm R}$-classes in
$S$. Similarly, the \defterm{left syntactic index} $[S:T]_{\rm L}$ of
$T$ in $S$ is the number of $\sigma_{\rm L}$-classes in $S$.
\end{definition}

In other words, $\sigma_{\rm R}$ and $\sigma_{\rm L}$ are the largest
right congruence and largest left congruence on $S$ that respect $T$
(that is, for which $T$ is a union of congruence classes).

The right syntactic index of $T$ in $S$ is finite if and only if the
left syntactic index of $T$ in $S$ is finite
\cite[Theorem~3.2(iii)]{ruskuc_syntactic}. It therefore makes sense to
state that a subsemigroup is of finite syntactic index. Notice further
that if $T$ is a large subsemigroup of $S$, then $T$ has
finite syntactic index in $S$ \cite[Corollary~4.4]{ruskuc_syntactic}.

The syntactic indices have an important advantage over the Rees index:
they are generalizations of the group index: by
\cite[Theorem~3.2(iii)]{ruskuc_syntactic}, if $G$ and $H$ are groups,
then
\[
[G:H] = [G:H]_{\rm R} = [G:H]_{\rm L}.
\]

However, the syntactic indices fail miserably when it comes to
the inheritance of finiteness properties by finite syntactic index
subsemigroups or extensions: any property of semigroups is either not
inherited by finite syntactic index subsemigroups or not inherited by
finite syntactic index extensions
\cite[Theorem~3.5]{ruskuc_syntactic}. The proof of this relies on a
semigroup with a zero adjoined, but even in the relatively
`group-like' situation of group-embeddable semigroups, finite
syntactic index subsemigroups and extensions do not inherit common
finiteness properties \cite[\S~9.3]{c_phdthesis}.

The most recent proposed alternative notion of index is the Green index, introduced by
Gray \& Ru\v{s}kuc \cite{gray_green1}:

\begin{definition}
Let $S$ be a semigroup and let
$T$ be a subsemigroup of $S$.
For $u,v \in S$ define:
\[
u \rel{R}^T v \iff uT^1 = vT^1, \quad u \rel{L}^T v \iff  T^1u = T^1v,
\]
and $\rel{H}^T = \rel{R}^T \cap \rel{L}^T$. Each of these relations is
an equivalence relation on $S$; their equivalence classes are called
the ($T$-)\defterm{relative} $\rel{R}$-, $\rel{L}$-, and
$\rel{H}$-classes, respectively. Furthermore, these relations respect
$T$, in the sense that each $\rel{R}^T$-, $\rel{L}^T$-, and
$\rel{H}^T$-class lies either wholly in $T$ or wholly in $S - T$. The
\defterm{Green index} of $T$ in $S$ is one more than the number of
relative $\rel{H}$-classes in $S \setminus T$.
\end{definition}

If $H$ is a finite-index subgroup of a group $G$, then $H$ has finite
Green index in $G$ \cite[Proposition~8]{gray_green1}. If $T$ is a
large subsemigroup of $S$, then $T$ has finite Green index in $S$
\cite[Proposition~6]{gray_green1}. In this sense, the Green index is a
common generalization of both the Rees and group indices.

The Green index seems more successful than the syntactic index in
terms of the inheritance of finiteness properties by finite Green
index subsemigroups or extensions. For example, finite Green index
subsemigroups and extensions inherit finiteness and more generally
local finiteness, periodicity, having finitely many right
(respectively, left) ideals \cite[Theorem~2(I--III)]{gray_green1},
finite generation \cite[Theorems~4.1 \&~4.3]{cgr_greenindex}, and
finite Malcev presentation
\cite[Theorem~7.1]{cgr_greenindex}. Automaticity is inherited by
finite Green index subsemigroups but not by finite Green index
extensions \cite[Theorem~10.1 \&
  Example~10.3]{cgr_greenindex}. Finitely presentability is not
inherited by finite Green index extensions; it is not known it is
inherited by finite Green index subsemigroups
\cite[Question~6.2(i)]{cgr_greenindex}.

\section{Generators and relations}

\subsection{Finite generation}
\label{sec:fingen}

The preservation of finite generation on passing to small extensions
is immediate: one can simply take a finite generating set for the
original semigroup and add all the elements in the complement to get a
finite generating set for the extension.

The other direction is less straightforward. Jura \cite{jura_ideals}
describes a generating set for a large subsemigroup that is sufficient
to show finite generation is inherited, but it seems to be a `dead
end': there seems to be no way to use it to obtain a presentation for
subsemigroups.

However, Campbell et al.~\cite[\S~3]{campbell_reidemeister} devised a
way to define a generating set for a large subsemigroup that can serve
to rewrite a presentation for a semigroup into a presentation for a
subsemigroup. Let us outline their ideas. Suppose $S$ is a semigroup
with a subsemigroup $T$. (For the present, do \emph{not} assume $T$ is
a large subsemigroup.) Let $A$ be an alphabet representing a
generating set for $S$. Define
\[
L(A,T) = \{w \in A^+ : w \in T\}.
\]
Choose a set $C \subseteq A^*$ such that every element of $S-T$ is
represented by a unique word in $C$, and every element of $C$
represents some element of $S-T$. For any $w \in S - T$, let $\rep{w}$
be the unique representative of $w$ in $C$.

Let $D$ be the alphabet
\[
\{d_{\rho, a,\sigma} : \rho,\sigma \in C \cup \{\emptyword\}, a \in A,
\rho a, \rho a\sigma \in L(A,T)\},
\]
and that, for all $\rho$, $a$, and $\sigma$, the letter $d_{\rho,
  a,\sigma}$ represents $\rho a\sigma$.

Define a mapping $\phi : L(A,T) \to D^+$ as follows. Let $w
\in L(A,T)$ with $w'a$ being the shortest prefix of $w$ lying in
$L(A,T)$ and $w''$ being the remainder of $w$. Then
\[
w\phi = \begin{cases}
d_{\rep{w'},a,\rep{w''}} & \text{ if $w'' \notin L(A,T)$,}\\
d_{\rep{w'},a,\emptyword}(w''\phi) & \text{ if $w'' \in L(A,T)$.}
\end{cases}
\]

This mapping $\phi$ rewrites words in $L(A,T)$ to words over $D$
representing the same element of $T$. In particular, every element of
$T$ is represented by some word over $D$; hence $D$ generates $T$
\cite[Theorem~3.1]{campbell_reidemeister}.

Now, if $A$ is finite and $T$ is a large subsemigroup, then $D$ is
finite. Together with the obvious fact that if $X$ generates $T$, then
$X \cup S-T$ generates $S$, this proves the following result:

\begin{theorem}
Finite generation is inherited by small extensions and by large
subsemigroups.
\end{theorem}

Jura \cite{jura_ideals} proved that finite generation is inherited by
large ideals using a different technique; the result for large
subsemigroups was first proved by Campbell et
al.~\cite[Corollary~3.2]{campbell_reidemeister}.

\subsection{Finite presentations}
\label{sec:finpres}

Proving that finite presentation is inherited by small extensions is
straightforward:

\begin{theorem}[{\cite[Theorems~4.1 \&~6.1]{ruskuc_largesubsemigroups}}]
\label{thm:finpres}
Finite presentation is inherited by small extensions and by large
subsemigroups.
\end{theorem}

The proof for small extensions proceeds as follows: Let $T$ be a
finitely presented semigroup and let $S$ be a small extension of
$T$. Let $\sgpres{A}{\rel{R}}$ be a finite presentation of $T$. For $a
\in A$ and $s,s' \in S-T$, fix words $\rho_{sa},\lambda_{as},\pi_{ss'}
\in A^* \cup (S-T)$ such that $sa =_S \rho_{sa}$, $as =_S
\lambda_{as}$, $ss' = \pi_{ss'}$. Let
\[
\rel{S} = \{sa = \rho_{sa}, as = \lambda_{sa}, ss' = \pi_{ss'} : a \in A, s,s' \in S-T\};
\]
notice that $\rel{S}$ is finite since $A$ and $S-T$ are finite. Then
one can show that $\sgpres{A \cup (S-T)}{\rel{R}\cup\rel{S}}$ is a
finite presentation for $S$.

The proof for large subsemigroups is much more complicated. Let $S$,
$T$, $A$, and $D$ be as in \fullref{\S}{sec:fingen}. Define another
mapping $\psi : D^+ \to L(A,T)$ by extending the mapping $d_{\rho,
  a,\sigma} \mapsto \rho a\sigma$ to $D^+$ in the natural way. Notice
that $w$ and $w\psi$ represent the same element of $T$.

\begin{theorem}[{\cite[Theorem~2.1]{campbell_reidemeister}}]
\label{thm:rewritingpres}
The subsemigroup $T$ is presented by $\sgpres{D}{\rel{Q}}$, where
$\rel{Q}$ contains the following infinite collection of defining
relations:
\begin{equation}
\label{eq:rewritingpres} 
\left.
\qquad
\qquad
\qquad
\begin{aligned}
(\rho a\sigma)\phi &= d_{\rho, a,\sigma}  \\
(w_1w_2)\phi &= (w_1\phi)(w_2\phi) \\
(w_3uw_4)\phi &= (w_3vw_4)\phi, 
\end{aligned}
\qquad
\qquad
\qquad
\right\}
\end{equation}
(where $\rho,\sigma \in C \cup \{\emptyword\}$, $a \in A$, $\rho
a,\rho a\sigma\in
L(A,T)$, $w_1,w_2 \in L(A,T)$, $w_3,w_4 \in A^*$, $(u,v) \in
\rel{P}$, $w_3uw_4 \in L(A,T)$). 
\end{theorem}

\fullref{Theorem}{thm:rewritingpres} in general gives an infinite
presentation for the subsemigroup $T$. Under the assumption that $T$
is a large subsemigroup of $S$, and that $S$ is finitely presented,
Ru{\v{s}}kuc proves that there is a finite set of relations $\rel{S}
\subseteq D^+ \times D^+$ (with $u =_T v$ for all $(u,v) \in \rel{S}$)
such that the relations \eqref{eq:rewritingpres} all lie in
$\cgen{\rel{S}}$ (that is, are all consequences of the relations in
$\rel{S}$); hence, under these assumptions, $T$ is finitely presented
by $\sgpres{D}{\rel{S}}$. Ru{\v{s}}kuc specifies the set $\rel{S}$ as
consisting of all relations in $D^+ \times D^+$ up to a certain length
that hold in $T$. The proof that this set $\rel{S}$ suffices is a very
long and intricate division into cases. [Gray \& Ru{\v{s}}kuc
  \cite[\S~5]{gray_boundaries} later observed and described how to fix
  a slight problem in one of the cases.]

\subsection{Cancellative and Malcev presentations}
\label{sec:cancpres}

Ordinary semigroup presentations define semigroups by means of
generators and defining relations. Informally, Malcev presentations
define semigroups by means of generators, defining relations, and a
rule of group-embeddability. Similarly, can\-cel\-la\-tive (respectively,
left-can\-cel\-la\-tive, right-can\-cel\-la\-tive) presentations define a
semigroup by means of generators, defining relations, and a rule of
can\-cel\-la\-tiv\-ity (respectively, left-can\-cel\-la\-tiv\-ity,
right-can\-cel\-la\-tiv\-ity). [Spehner \cite{spehner_presentations}
  introduced Malcev presentations and named them for Malcev's
  group-embeddability condition \cite{malcev_immersion1}; Croisot
  \cite{croisot_automorphisms} introduced can\-cel\-la\-tive presentations;
  Adjan \cite{adjan_defining} introduced left-can\-cel\-la\-tive and
  right-can\-cel\-la\-tive presentations.] 

Spehner and Adjan showed that a rule of group-embeddability,
can\-cel\-la\-tiv\-ity, left-can\-cel\-la\-tiv\-ity, or right-can\-cel\-la\-tiv\-ity is worth
an infinite number of defining relations, in the sense that a finitely
generated semigroup may admit a finite Malcev presentation, but no
finite can\-cel\-la\-tive presentation
(see~\cite[Theorem~3.4]{spehner_presentations}); a finite can\-cel\-la\-tive
presentation, but no finite left- or right-can\-cel\-la\-tive presentation
(see~\cite[Theorem~3.1(ii)]{spehner_presentations}
and~\cite[Theorem~I.4]{adjan_defining}); a finite left-can\-cel\-la\-tive
presentation, but no finite `ordinary' or right-can\-cel\-la\-tive
presentation (see~\cite[Theorem~3.1(i)]{spehner_presentations}
and~\cite[Theorem~I.2]{adjan_defining}); a finite right-can\-cel\-la\-tive
presentation, but no finite `ordinary' or left-can\-cel\-la\-tive
presentation (see~\cite[Theorem~3.1(i)]{spehner_presentations}
and~\cite[Theorem~I.2]{adjan_defining}). [For further background
  information on Malcev presentations, see the survey \cite{c_malcev}.]

Let us now formally define Malcev and (left-/right-)can\-cel\-la\-tive presentations.

\begin{definition}
Let $S$ be any semigroup. A congruence $\sigma$ on $S$ is:
\begin{itemize}
\item a \defterm{Malcev congruence} if $S/\sigma$ is embeddable in a
  group.
\item a \defterm{can\-cel\-la\-tive congruence} if $S/\sigma$ is a
  can\-cel\-la\-tive semigroup.
\item a \defterm{left-can\-cel\-la\-tive congruence} if $S/\sigma$ is a
  left-can\-cel\-la\-tive semigroup.
\item a \defterm{right-can\-cel\-la\-tive congruence} if $S/\sigma$ is a
  right-can\-cel\-la\-tive semigroup.
\end{itemize}
\end{definition}

If $\{\sigma_i : i \in I\}$ is a set of Malcev (respectively,
can\-cel\-la\-tive, left-can\-cel\-la\-tive, right-can\-cel\-la\-tive) congruences on
$S$, then $\sigma = \bigcap_{i \in I} \sigma_i$ is also a Malcev
(respectively, can\-cel\-la\-tive, left-can\-cel\-la\-tive, right-can\-cel\-la\-tive)
congruence on $S$ (see \cite[Proposition~1.2.2]{c_phdthesis} and
\cite[Lemma~9.49]{clifford_semigroups1}).

\begin{definition}
Let $A^+$ be a free semigroup; let $\rho \subseteq A^+ \times A^+$ be
any binary relation on $A^+$. Let $\malcgen{\rho}$ denote the smallest
Malcev congruence containing $\rho$:
\[
\malcgen{\rho} = \bigcap \left\{\sigma : \sigma \supseteq \rho, \text{
  $\sigma$ is a Malcev congruence on }A^+\right\}.
\]
Then $\sgmpres{A}{\rho}$ is a \defterm{Malcev presentation} for [any
  semigroup isomorphic to] $A^+\!/\malcgen{\rho}$. 

Similarly, let $\ccgen{\rho}$ (respectively $\lccgen{\rho}$,
$\rccgen{\rho}$) denote the smallest can\-cel\-la\-tive (respectively,
left-can\-cel\-la\-tive, right-can\-cel\-la\-tive) congruence containing
$\rho$. Then $\sgcpres{A}{\rho}$ (respectively, $\sglcpres{A}{\rho}$,
$\sgrcpres{A}{\rho}$) is a \defterm{can\-cel\-la\-tive (respectively,
  left-can\-cel\-la\-tive, right-can\-cel\-la\-tive) presentation} for [any
  semigroup isomorphic to] $A^+\!/\ccgen{\rho}$ (respectively,
$A^+\!/\lccgen{\rho}$, $A^+\!/\rccgen{\rho}$).
\end{definition}

\begin{theorem}[{\cite[Theorems~1--3]{crr_finind}}]
Within the class of group-embeddable (respectively, can\-cel\-la\-tive,
left-can\-cel\-la\-tive, right-can\-cel\-la\-tive) semigroup, finite Malcev
(respectively, can\-cel\-la\-tive, left-can\-cel\-la\-tive, right-can\-cel\-la\-tive)
presentation is inherited by small extensions and by large subsemigroups.
\end{theorem}

In each case, the proof that the property is inherited by small
extensions can be proved using reasoning parallel to that described
above for \fullref{Theorem}{thm:finpres}. In contrast, the proof for
inheritance by large subsemigroups requires different proof techniques
for each type of presentation. First, the result for Malcev
presentations does not use any kind of rewriting technique: the key to
proving it is to use the pigeon-hole principle to deduce that if $S$
is a group-embeddable semigroup and $T$ is a large subsemigroup of
$S$, then every element of $S-T$ can be expressed as both a right
quotient and a left quotient of elements of $T$
\cite[Lemma~3.2]{crr_finind}. (These quotients are formed in the
universal group $G$ of $S$. The universal group of $S$ is the largest
group into which $S$ embeds and which $S$ generates; see
\cite[Ch.~12]{clifford_semigroups2}.) From this, it follows that the
universal groups of $S$ and $T$ are isomorphic
\cite[Theorem~3.1]{crr_finind}. A Malcev presentation for a semigroup
is essentially a presentation for its universal group
\cite[Proposition~1.3.1]{c_phdthesis}, and so the result follows.

However, the proofs for can\-cel\-la\-tive and left-/right-can\-cel\-la\-tive
presentations do make use of the rewriting techniques of Campbell et
al. Let $S$ be can\-cel\-la\-tive or left-can\-cel\-la\-tive as appropriate (the
reasoning for right-can\-cel\-la\-tive is dual) and let $T$ be a large
subsemigroup of $S$. Suppose $S$ admits a finite can\-cel\-la\-tive
presentation $\sgcpres{A}{\rel{R}}$ or finite left-can\-cel\-la\-tive
presentation $\sglcpres{A}{\rel{R}}$. The proofs make use of certain
syntactic rules that show when two words are equal in a semigroup
defined by a (left-/right-)can\-cel\-la\-tive presentation. Essentially
these syntactic rules allow the application of the defining relations
(just as in ordinary semigroup presentations), together with insertion
and deletion of words $\mall{u}u$ and $u\malr{u}$ under certain
restrictions. [The maps $u \mapsto \mall{u}$ and $u \mapsto \malr{u}$
  extend bijections from $A$ onto alphabets $\mall{A}$ and $\malr{A}$,
  and function in a similar way to element--inverse pairs, but their
  use is restricted.]

The proofs use the reasoning used to establish the large subsemigroups
part of \fullref{Theorem}{thm:finpres} in a purely syntactic way, in
the sense that it guarantees the existence of a finite set of
relations that have certain other relations as consequences. Extra
work is necessary to ensure that when we insert and delete word
$\mall{u}u$ and $u\malr{u}$, the word $u$ represents an element of the
large subsemigroup $T$. In the can\-cel\-la\-tive case, this is possible by
reducing to proving the result for large ideals by considering the
largest ideal $I$ contained in $T$, which also has finite Rees index
in $S$ as a consequence of can\-cel\-la\-tiv\-ity
\cite[Theorem~7.1]{crr_finind}. In the right-can\-cel\-la\-tive case, we
pass to the largest right ideal $K$, which again has finite Rees index
in $S$ \cite[Theorem~8.1]{crr_finind}, and then split into the cases
where (i) every element of $S-K$ has a right-multiple lying in $K$, or
(ii) $S-K$ is a subgroup and a right ideal of $S$. In case (i), the
same reasoning works as in the can\-cel\-la\-tive case; in case (ii), $S$
must be finite and the result holds trivially.

\subsection{Inverse semigroup presentations}
\label{sec:invpres}

An inverse semigroup presentation $\invpres{A}{\rel{R}}$ presents the
semigroup
\[
(A \cup A^{-1})^+/(\rel{R} \cup \rho)^\#,
\]
where $A^{-1}$ is an alphabet in bijection with $A$ under the map $a
\mapsto a^{-1}$ and
\begin{align*}
\rho ={}& \{(uu^{-1}u,u) : u \in (A \cup A^{-1})^+\} \\
&\cup \{(uu^{-1}vv^{-1},vv^{-1}uu^{-1}) : u,v \in (A \cup A^{-1})^+\}.
\end{align*}
Essentially, the relations $\rho$ ensure that every element has an
inverse and that idempotents commute, which is one of the
characterizations of an inverse semigroup
\cite[Theorem~5.1.1]{howie_fundamentals}. Note further that $\rho$ is infinite, so a
finite inverse semigroup presentation cannot be converted into a
finite semigroup presentation. Furthermore, there exist inverse
semigroups that admit finite inverse semigroup presentations but which
are not finitely presented.

A narrowly circulated preprint from around a decade ago contained a
theorem stating that if $S$ is an inverse semigroup and $T$ is an
inverse subsemigroup of finite Rees index in $S$, then $S$ has a
finite inverse semigroup presentation if and only if $T$ has a finite
inverse semigroup
presentation~\cite[Main~Theorem]{araujo_presentations}. This preprint
has never, to our knowledge, been published or given general
circulation, so we are uncertain if we can view as settled the
question of whether having a finite inverse semigroup presentation is
preserved on passing to large subsemigroups or small extensions
(within the class of inverse semigroups).

\subsection{Soluble word problem}
\label{sec:wordproblem}

\begin{theorem}[{\cite[Theorem~5.1(ii)]{ruskuc_largesubsemigroups}}]
Having soluble word problem is inherited by small extensions and by large
subsemigroups.
\end{theorem}

The class of semigroups with soluble word problem is closed under
forming finitely generated subsemigroups, so in one direction this
result is immediate. In the other direction, let $A$ be an alphabet
representing a generating set for $T$. Then $S$ is generated by $A
\cup (S-T)$. Using the finitely many extra relations added in the
small extension part proof of \fullref{Theorem}{thm:finpres}, one can
start from a word $u \in (A \cup (S-T))^+$ and effectively obtain a
word $u' \in A^+ \cup (S-T)$ representing the same element of
$S$. This reduces the word problem of $S$ to the word problem of $T$
plus checking equality within the finite set $S-T$.

\subsection{Finite complete rewriting systems}
\label{sec:fcrs}

A semigroup presentation can be naturally viewed as a rewriting
system. The most computationally friendly situation is when this
rewriting system is finite and complete. (Recall that a rewriting
system is complete if it is confluent and noetherian.) For further
background on rewriting systems, see~\cite{book_srs}.

\begin{theorem}[{\cite[Theorem~1]{wang_fcrsfdt} \& \cite[Theorem~1.1]{wong_fcrs}}]
Being presented by a finite complete rewriting system is inherited by
small extensions and by large subsemigroups.
\end{theorem}

The proof for small extensions, due to Wang
\cite[Theorem~1]{wang_fcrsfdt}, is a natural strengthening of the
small extensions part of the proof of
\fullref{Theorem}{thm:finpres}. Retaining notation from that proof, if
$\sgpres{A}{\rel{R}}$ is a finite complete rewriting system, then the
presentation $\sgpres{A}{\rel{R} \cup \rel{S}}$ for $S$ is also a
finite complete rewriting system. The most complicated part is proving
that $\sgpres{A}{\rel{R} \cup \rel{S}}$ forms a noetherian rewriting
system, which involves defining a rather intricate well-order on $(A
\cup (S-T))^*$ such that rewriting always decreases a word with
respect to this order.

The proof for large subsemigroups, due to Wong \& Wong
\cite[Theorem~1.1]{wong_fcrs}, is much more difficult. It builds upon
and strengthens the strategy of the large subsemigroups part of the
proof of \fullref{Theorem}{thm:finpres}. Retain the notation of that
proof and suppose further that $\sgpres{A}{\rel{R}}$ be a finite
complete rewriting system presenting $S$.  Let $K(A,T)$ be a subset of
$L(A,T)$ that includes all irreducible words representing elements of
$T$. (Recall from \fullref{\S}{sec:fingen} that $L(A,T)$ is the
language of all words over $A$ representing elements of $T$.)  Let
$\sgpres{B}{\rel{S}}$ be a finite complete rewriting system, $\psi :
B^+ \to L(A,T)$ a homomorphism, and $\phi : K(A,T) \to B^+$ a map. A
six-part technical condition is defined on the tuple
$(\sgpres{B}{\rel{S}},K(A,T),\psi,\phi)$. If the tuple satisfies this
condition, then $\sgpres{B}{\rel{S}}$ presents $T$. Part of the
technical condition is that $u\psi\phi = u$ for all $u \in B^+$; thus
$\psi$ and $\phi$ play a similar role to the corresponding maps in
\fullref{\SS}{sec:fingen} and~\ref{sec:finpres}.

A long argument, ultimately using the condition in the last paragraph
when in the case $T=S$, allows one to construct a finite complete
rewriting system $\sgpres{A'}{\rel{R}'}$ where every element of $S-T$
is represented by a symbol in $A'$. Equipped with this new rewriting
system for $S$, one can define the alphabet $B$, the relation
$\rel{S}$, the maps $\phi$ and $\psi$, and the set $K(A,T)$ and show
that the technical condition is satisfied and that thus
$\sgpres{B}{\rel{S}}$ is a finite complete rewriting system presenting
$T$.

\section{Homological and cohomological conditions}
\label{sec:homolo}

\subsection{Finite derivation type}
\label{sec:fdt}

Consider a semigroup presentation $\sgpres{A}{\rel{R}}$. The
\defterm{derivation graph} of this presentation is the infinite graph
$\Gamma = (V,E,\iota, \tau, ^{-1})$ with vertex set $V =
A^*$, and edge set $E$ consisting of the collection of $4$-tuples
\[
\{ (w_1, r, \epsilon, w_2): \ w_1, w_2 \in A^*, r \in \rel{R}, \ \mbox{and} \ \epsilon \in \{ +1, -1 \} \}.
\]
The functions $\iota, \tau : E \to V$ map an edge $\bbe = (w_1, r,
\epsilon, w_2)$ (with $r=(r_{+1},r_{-1}) \in \rel{R}$) to its initial
and terminal vertices $\iota \bbe = w_1 r_{\epsilon} w_2$ and $\tau
\bbe = w_1 r_{- \epsilon} w_2$, respectively. The mapping $^{-1} : E
\rightarrow E$ maps an edge $\bbe = (w_1, r, \epsilon, w_2)$ to its
inverse edge $\bbe^{-1} = (w_1, r, -\epsilon, w_2)$.

A path is a sequence of edges $\bbp = \bbe_1 \circ \bbe_2 \circ \ldots
\circ \bbe_n$ where $\tau \bbe_i = \iota \bbe_{i+1}$ for $i=1,
\ldots, {n-1}$. Here $\bbp$ is a path from $\iota \bbe_1$ to $\tau
\bbe_n$ and we extend the mappings $\iota$ and $\tau$ to paths by
defining $\iota \bbp = \iota \bbe_1$ and $\tau \bbp = \tau
\bbe_n$.  The inverse of a path $\bbp = \bbe_1 \circ \bbe_2 \circ
\ldots \circ \bbe_n$ is the path $\bbp^{-1} = \bbe_n^{-1} \circ
\bbe_{n-1}^{-1} \circ \ldots \circ \bbe_1^{-1}$, which is a path from
$\tau \bbp$ to $\iota \bbp$. A \defterm{closed path} is a path $\bbp$
satisfying $\iota \bbp = \tau \bbp$. For two paths $\bbp$ and
$\bbq$ with $\tau \bbp = \iota \bbq$ the composition $\bbp \circ
\bbq$ is defined.

We denote the set of paths in $\Gamma$ by $P(\Gamma)$, where for each
vertex $w \in V$ we include a path $1_w$ with no edges, called the
\emph{empty path} at $w$. The free monoid $A^*$ acts on both sides of
the set of edges $E$ of $\Gamma$ by
\[
x \cdot \bbe \cdot y = (x w_1, r, \epsilon, w_2 y)
\]
where $\bbe = (w_1, r, \epsilon, w_2)$ and $x, y \in A^*$. This
extends naturally to a two-sided action of $A^*$ on $P(\Gamma)$ where
for a path $\bbp = \bbe_1 \circ \bbe_2 \circ \ldots \circ \bbe_n$ we
define
\[
x \cdot \bbp \cdot y = (x \cdot \bbe_1 \cdot y) \circ (x \cdot \bbe_2 \cdot y)
\circ \ldots \circ (x \cdot \bbe_n \cdot y).
\]
If $\bbp$ and $\bbq$ are paths such that $\iota \bbp = \iota
\bbq$ and $\tau \bbp = \tau \bbq$ then $\bbp$ and $\bbq$ are
\defterm{parallel}, denotes $\bbp \parallel \bbq$.

An equivalence relation $\sim$ on $P(\Gamma)$ is called a
\emph{homotopy relation} if it is contained in $\parallel$ and
satisfies the following four conditions.


\begin{itemize}

\item[(H1)] If $\bbe_1$ and $\bbe_2$ are edges of $\Gamma$, then
\[
(\bbe_1 \cdot \iota \bbe_2) \circ (\tau \bbe_1 \cdot \bbe_2) \sim
(\iota \bbe_1 \cdot \bbe_2) \circ (\bbe_1 \cdot \tau \bbe_2 ).
\]

\item[(H2)] For any $\bbp, \bbq \in P(\Gamma)$ and $x,y \in A^*$
\[
\bbp \sim \bbq \implies x \cdot \bbp \cdot y \sim x \cdot \bbq \cdot y.
\]

\item[(H3)] For any $\bbp, \bbq, \bbr, \bbs \in P(\Gamma)$ with $\tau
  \bbr = \iota \bbp = \iota \bbq$ and $\iota \bbs =
  \tau \bbp = \tau \bbq$
\[
\bbp \sim \bbq \implies \bbr \circ \bbp \circ \bbs \sim \bbr \circ \bbq \circ \bbs.
\]

\item[(H4)] If $\bbp \in P(\Gamma)$ then $\bbp \bbp^{-1} \sim 1_{\iota
  \bbp}$, where $1_{\iota \bbp}$ denotes the empty path at the vertex
  $\iota \bbp$.

\end{itemize}

For a subset $C$ of $\parallel$, the homotopy relation \emph{$\sim_C$
  generated by $C$} is the smallest (with respect to inclusion)
homotopy relation containing $C$. If $\sim_C$ coincides with
$\parallel$, then $C$ is called a \emph{homotopy base} for
$\Gamma$. The semigroup presented by $\sgpres{A}{\rel{R}}$ is said to
have \emph{finite derivation type} (\fdt) if the derivation graph of
$\pres{A}{\rel{R}}$ admits a finite homotopy base. \fdt\ is
independent of the choice of finite presentation
\cite[Theorem~4.3]{squier_finiteness}.

\begin{theorem}[{\cite[Theorem~2]{wang_fcrsfdt}}]
Finite derivation type is inherited by small extensions.
\end{theorem}

The essence of the proof is as follows: The semigroup $S$ has a finite
presentation $\pres{A \cup (S-T)}{\rel{R} \cup \rel{S}}$, where
$\pres{A}{\rel{R}}$ is a presentation for $T$. Take a finite homotopy
basis for $\pres{A}{\rel{R}}$. Add finitely many parallel paths to the
basis to generate pairs of parallel paths $\bbe \parallel \bbp\circ
\bbq\circ \bbp'$, where $\bbe$ is an edge between vertices in $(A \cup
(S-T))^*$, the paths $\bbp$ and $\bbp'$ are, respectively, from
$\iota\bbe$ and $\tau\bbe$ to some vertices in $A^* \cup (S-T)$, and
$\bbq$ is a path in $A^* \cup (S-T)$. Then this new basis in fact
generates all parallel paths in the derivation graph for $\pres{A \cup
  (S-T)}{\rel{R} \cup \rel{S}}$.

\begin{question}
\label{qu:fdt}
Is finite derivation type inherited by large subsemigroups?
\end{question}

Although this question remains open in general, the special case of
large ideals has been settled:

\begin{theorem}[{\cite[Theorem~1]{malheiro_fdt}}]
Finite derivation type is inherited by large ideals.
\end{theorem}

The proof of this result uses the following rewriting theorem:

\begin{theorem}[{\cite[Theorem~4.1]{malheiro_trivializers}}]
\label{thm:rewritinghomobasis}
Retain notation from \fullref{Theorem}{thm:rewritingpres}. For any $v
\in D^*$, let $\Lambda_v$ be a path from $v$ to $v\phi\psi$ in the
derivation graph of $\sgpres{D}{\rel{Q}}$. (Such paths always exist.)
If $X$ is a homotopy basis of $\sgpres{A}{\rel{R}}$, then a homotopy
basis for $\sgpres{D}{\rel{Q}}$ is the set of all parallel
paths of the form
\begin{equation}
\label{eq:rewritinghomobasis}
\left.
\qquad
\qquad
\qquad
\begin{aligned}
(\bbe \circ \Lambda_{\tau\bbe},{}&\Lambda_{\iota\bbe} \circ \bbe\psi\phi), \\
(\bbe_1\circ \bbe_2, {}& \bbe_2\circ \bbe_1)\phi, \\
(w_1\cdot \bbp\cdot w_2,{}& w_1\cdot \bbq\cdot w_2)\phi,
\end{aligned}
\qquad
\qquad
\qquad
\right\} 
\end{equation}
where $\bbe$ is any edge in the derivation graph of
$\sgpres{D}{\rel{Q}}$; and $\bbe_1,\bbe_2$ are any edges in the
derivation graph of $\sgpres{A}{\rel{R}}$ such that
$\iota\bbe_1\iota\bbe_2 \in L(A,T)$; and $(\bbp,\bbq) \in X$ and
$w_1,w_2 \in A^*$ are such that $w_1\iota\bbp w_2 \in L(A,T)$.
\end{theorem}

Just as \fullref{Theorem}{thm:rewritingpres} always gives an infinite
presentation for the subsemigroup,
\fullref{Theorem}{thm:rewritinghomobasis} always gives an infinite
homotopy basis for the derivation graph of a presentation of the
subsemigroup. However, when the subsemigroup is a finite Rees index
ideal, it is possible to find a finite set of parallel paths that
generate all the parallel paths \eqref{eq:rewritinghomobasis} and so
forms a finite homotopy basis for the ideal and so proving the result
\cite[\S~8]{malheiro_fdt}.

\subsection{\texorpdfstring{$\FP_n$}{FPn}}
\label{sec:fpn}

Let $S$ be a monoid and let $\zset S$ be the integral monoid ring of
$S$. We can regard $\zset$ as a trivial left $\zset S$-module with the
$\zset S$-action via the standard augmentation $\epsilon_S : \zset S
\to \zset$ (where $s \mapsto 1$ for all $s \in S$): that is, $w \cdot
z = \epsilon_S(w)z$ for $w \in \zset S$ and $z \in \zset$. A free
resolution of the trivial left $\zset S$-module $\zset$ is an exact
sequence
\[
\ldots \stackrel{\delta_3}{\to} A_2 \stackrel{\delta_2}{\to} A_1 \stackrel{\delta_1}{\to} A_0 \stackrel{\delta_0}{\to} \zset \to 0,
\]
in the sense that $\delta_0 : A_0 \to \zset$ and $\delta_i : A_i \to
A_{i-1}$, for $i \geq 1$, are homomorphisms such that $\im\delta_{i} =
\ker\delta_{i-1}$ for $i \geq 1$, and $\im\delta_0 = \zset$, in which
$A_0, A_1, A_2,\ldots$ are free left $\zset S$-modules and
homomorphisms. A monoid $S$ has property \defterm{left-$\FP_n$} if
there is a partial free resolution of the trivial left $\zset
S$-module $\zset$
\[
A_n \to A_{n-1} \to \ldots \to A_1 \to A_0 \to \zset \to 0
\]
where $A_0, A_1, \ldots, A_n$ are all finitely generated; if there is
a (non-partial) free resolution where all the $A_i$ are finitely
generates, the monoid is has property $\FP_\infty$. Note that if $n \geq m$,
then left-$\FP_n$ implies right $\FP_m$. Right-$\FP_n$
\defterm{right-$\FP_n$} is defined dually. For groups, left-$\FP_n$ is
equivalent to right-$\FP_n$, but there are completely independent for
semigroups, in the sense that exists a monoid that is
right-$\FP_\infty$ but not left-$\FP_1$.

\begin{theorem}[{\cite[\S~8]{gray_homological}}]
The properties left- and right-$\FP_n$ (for any $n \in \nset \cup
\{\infty\}$) are not inherited in general by large subsemigroups.
\end{theorem}

A monoid with a zero is left- and right-$\FP_\infty$
\cite[Proposition~3.1]{kobayashi_homological}. Thus if $G$ is a group
that is not $\FP_1$ (and thus not left- or right-$\FP_n$ for any $n
\in \nset \cup \{\infty\}$), then it is a large subsemigroup of $G^0$,
which is left- and right-$\FP_\infty$ (and thus left- and
right-$\FP_n$ for any $n \in \nset \cup \{\infty\}$).

\begin{question}
Are left- and right-$\FP_n$ (for $n \in \nset \cup \{\infty\}$
inherited by small extensions?
\end{question}

\subsection{Finite cohomological dimension}
\label{sec:fcd}

A projective resolution of the trivial left $\zset S$-module $\zset$ is an exact sequence
\[
\ldots \stackrel{\delta_3}{\to} P_2 \stackrel{\delta_2}{\to} P_1 \stackrel{\delta_1}{\to} P_0 \stackrel{\delta_0}{\to} \zset \to 0,
\]
in which the $P_i$ are projective left $\zset S$-modules. Such
resolutions always exist. If there exists $n \geq 0$ such that $P_n
\neq 0$ but $P_i = 0$ for all $i > n$, then the resolution has length
$n$, and otherwise has infinite length. The minimum length among all
projective resolutions of $\zset$ is the left cohomological dimension
of $S$. The right cohomological dimension is defined dually. For
groups, the left and right cohomological dimensions coincide.

\begin{theorem}[{\cite[\S~8]{gray_homological}}]
Finite left and right cohomological dimension are not inherited in
general by large subsemigroups.
\end{theorem}

A monoid with a zero has left and right cohomological dimension~$0$
\cite[Theorem~1]{guba_cohomological}. Thus if $G$ is a group with
infinite cohomological dimension, then it is a large subsemigroup of
$G^0$, which is a monoid with finite left and right cohomological
dimension.

\begin{question}
Are finite left and right cohomological dimension inherited by
small extensions?
\end{question}

\section{Residual finiteness}
\label{sec:residualfin}

\begin{theorem}[{\cite[Corollary~4.6]{ruskuc_syntactic}}]
Residual finiteness is inherited by small extensions and by large
subsemigroups.
\end{theorem}

The result for large subsemigroups is trivial: residual finiteness is
inherited by arbitrary subsemigroups. So let $T$ be a residually
finite semigroup and let $S$ be a small extension of $S$. Let $s,t \in
S$. If $s,t \in T$, then since $T$ is residually finite, there is a
congruence $\eta$ on $T$, with finitely many classes, that separates
$s$ and $t$. Let $\zeta$ be the largest congruence contained in $\eta
\cup \Delta_{S-T}$, where $\Delta_{S-T} = \{(x,x) : x \in S-T\}$. Then
$\zeta$ has finitely many classes \cite[Theorems~2.4 
  and~4.3]{ruskuc_syntactic} and separates $s$ and $t$. If at least
one of $s$ and $t$ lies outside $T$, let $\zeta$ be the largest
congruence contained in $(T \times T) \cup \Delta_{S-T}$. Then again
$\zeta$ has finitely many classes and separates $s$ and $t$.

\section{Periodicity and related properties}
\label{sec:periodicity}

\subsection{Local finiteness, local finite presentation, and periodicity}
\label{sec:localfin}

\begin{theorem}[{\cite[Theorem~5.1(ii--iv)]{ruskuc_largesubsemigroups}}]
Local finiteness, being locally finitely presented, and periodicity
are all inherited by small extensions and by large subsemigroups.
\end{theorem}

For large subsemigroups, this result is trivial: local finiteness,
being locally finitely presented, and periodicity are all inherited by
arbitrary subsemigroups.

The proofs for small extensions are quite short. Let $T$ be a
semigroup and $S$ a small extension of $T$. The proofs for local
finiteness and being locally finitely presented proceed in the same
way: Suppose $T$ is locally finite (respectively, locally finitely
presented). Let $U$ be a finitely generated subsemigroup of $S$. Then
$U \cap T$ is finite (respectively, finitely presented). Since $U$ is
a small extension of $T$, it follows that $U$ is finite (respectively,
finitely presented). Since $U$ was arbitrary, the results follow.

Suppose $T$ is periodic. Let $s \in S$. Consider $s^i$ for all $i$. If
all these lie in $S-T$, some must be equal by the pigeon-hole
principle. If $s^k \in T$, then periodicity in $T$ applies to show $s$
is periodic. Hence $S$ is periodic.

(The more general result for finite Green index extensions is also
known to hold for local finiteness and periodicity
\cite[Theorem~2(I--II)]{gray_green1}.)

\subsection{Global torsion}
\label{sec:globaltorsion}

A semigroup $S$ has \defterm{global torsion} if $S^n = S^{n+1}$ for
some $n \in \nset$. 

\begin{theorem}[{\cite[Theorem~8.1]{gray_ideals}}]
Global torsion is inherited by small extensions but not in general by
large subsemigroups.
\end{theorem}

Global torsion is actually inherited by finite Green index extensions
\cite[Theorem~8.1]{gray_ideals}.

To construct a counterexample for large subsemigroups, let $T$ be a
semigroup without global torsion (such as a non-trivial null
semigroup). Then it is a large subsemigroup of $T^1$, which has global
torsion since $(T^1)^2 \supseteq 1T^1 = T^1$.

\subsection{Eventual regularity}
\label{sec:eventualreg}

An element $s$ of a semigroup if \defterm{eventually regular} if $s^n$
is regular for some $n \in \nset$. A semigroup is \defterm{eventually
  regular} if all its elements are eventually regular. Eventual
regularity was introduced by Edwards~\cite{edwards_eventually} and has
received a lot of attention.

\begin{theorem}[{\cite[Theorem~9.2]{gray_ideals}}]
Eventual regularity is inherited by small extensions and by large
subsemigroups.
\end{theorem}

Both parts of this result actually hold for finite Green
index~\cite[Theorem~9.2]{gray_ideals}.

\section{Hopficity and co-hopficity}
\label{sec:hopf}

Recall that an algebraic or relational structure is \emph{hopfian} if
every surjective endomorphism of that structure is bijective and thus
an automorphism. Dually, a structure is \emph{co-hopfian} if every
injective endomorphism of that structure is bijective and thus an
automorphism.

\begin{theorem}[{\cite[\S~2]{maltcev_hopfian}}]
Hopficity is not inherited in general by large subsemigroups
or by small extensions
\end{theorem}

The proof of this result is straightforward: Define a family of
isomorphic semigroups $T_i=\pres{b_i}{b_i^2=b_i^4}$, $i\in\nset$. Form
their disjoint union $T$, and extend the multiplication defined on
each $T_i$ to a multiplication on the whole of $T$ by letting
$xy=yx=y$ for any $x\in T_i$ and $y\in T_j$, where $i<j$. It is easy
to see that this multiplication is associative.

Let $F$ be the semigroup $\pres{a}{a^5=a^2}$. Let $S=T\cup F$, and
extend the multiplication on $T$ and $F$ to a multiplication on the
whole of $S$ by $xy=yx=y$ for $x\in F$ and $y\in T$. Again, this turns
$S$ into a semigroup. Notice that $T\subseteq T^1\subseteq S^1$ is a
chain of small extensions. It is straightforward to show that $S^1$
and $T$ are hopfian but $T^1$ is not
\cite[Proposition~2.1]{maltcev_hopfian}.

However, adding an assumption of finite generation yields a positive
result for small extensions, although not for large subsemigroups:

\begin{theorem}[{\cite[Main~Theorem]{maltcev_hopfian}}]
\label{thm:hopfian}
Within the class of finitely generated semigroup, hopficity is
inherited by small extensions but not in general by large subsemigroups.
\end{theorem}

In fact, the positive result for small extensions is a consequence of
the following theorem, which is independently interesting:

\begin{theorem}[{\cite[Theorem~3.1]{maltcev_hopfian}}]
Let $T$ be a finitely generated semigroup and let $S$ be a proper
small extension of $T$ (that is, a small extension such that $S \neq
T$). Let $\phi$ be an endomorphism of $S$. Then $T\phi\neq S$.
\end{theorem}

For large subsemigroups, the first counterexample depended on the
notion of semigroup actions, and connected the hopficity of an action
of a free semigroup to the hopficity of a semigroup constructed from
that action \cite[\S~5]{maltcev_hopfian}. An improved counterexample,
also showing that hopficity is not inherited by large subsemigroups
even within the class of finitely presented semigroup, is the
following:

\begin{example}[{\cite[Example~3.1]{cm_hopf}}]
\label{eg:hopffreesfpdown}
Let 
\begin{align*}
S &= \sgpres{a,b,f}{abab^2ab = b, fa=ba, fb = bf = f^2 = b^2},\\
T &= \sgpres{a,b}{abab^2ab = b}.
\end{align*}
Then $T$ is a large subsemigroup of $S$ since $S = T \cup \{f\}$.

Let $\psi : S \to S$ be a surjective endomorphism. Since $a$ and $f$
are the only indecomposable elements of $S$, we have $\{a,f\}\psi =
\{a,f\}$. Let $\vartheta = \psi^2$; then $\vartheta$ is a surjective
endomorphism of $S$ with $a\vartheta = a$ and $f\vartheta = f$.

If $b\vartheta = f$, then $f =_S b\vartheta =_S (abab^2ab)\vartheta =_S
afaf^2af =_S abab^2ab =_S b$, which is a contradiction. Hence $b\vartheta = w \in T$. Then
\[
ab =_S af = (a\vartheta)(f\vartheta) =_S (af)\vartheta =_S (ab)\vartheta =
(a\vartheta)(b\vartheta) = aw.
\]
But $S$ is left-cancellative by Adjan's theorem
\cite{adjan_defining}; hence $b =_S w$. That is, $b\vartheta = w
=_S b$. Since $a\vartheta = a$ and $f\vartheta = f$, the endomorphism
$\vartheta$ must be the identity mapping on $S$ and so
bijective. Hence $\psi$ is bijective and so an automorphism. This
proves that $S$ is hopfian.

Now the aim is to show that $T$ is non-hopfian. Notice that
\[
abab^3 =_T abab^2(abab^2ab) = (abab^2ab)ab^2 ab =_T bab^2ab.
\]
It easy to check that the rewriting system $(\{a,b\},\{abab^2ab
\imreduces b, abab^3 \imreduces bab^2ab\})$ is confluent and
noetherian. Clearly $T$ is also presented by this rewriting system.

The map
\[
\phi : T \to T; \qquad\qquad a \mapsto a,\qquad b\mapsto bab.
\]
is a well-defined endomorphism, and is surjective since $a\phi = a$ and
\begin{equation}
\label{eq:hopffreesfpdown1}
(ab^2)\phi =_T a(bab)^2 =_T abab^2ab \imreduces b.
\end{equation}
Furthermore, applying
\eqref{eq:hopffreesfpdown1} shows that
\[
(ab^2a^2b^2)\phi =_T (ab^2\;a\;ab^2)\phi =_T bab =_T b\phi.
\]
But both $ab^2a^2b^2$ and $b$ are irreducible and so $ab^2a^2b^2
\neq_T b$. Hence $\phi$ is not bijective and so $T$ is not hopfian.
\end{example}

The situation for co-hopficity mirrors that for hopficity:

\begin{theorem}[{\cite[\S~4]{cm_hopf}}]
Co-hopficity is not inherited in general by large subsemigroups
or by small extensions
\end{theorem}

The following counterexample establishes the result for large
subsemigroups:

\begin{example}[{\cite[Example~4.1]{cm_hopf}}]
\label{eg:cohopfreesdown}
Let $T$ be the free semigroup with basis $x$. Then any map $x \mapsto
x^k$ extends to an injective endomorphism from $T$ to itself; for $k
\geq 2$ this endomorphism is not bijective and so not an
automorphism. Thus $T$ is not co-hopfian.

Let
\[
S = \sgpres{x,y}{y^2 = xy = yx = x^2}
\]
Notice that $S = T \cup \{y\}$ since all products of two or more
generators must lie in $T$. So $S$ is a small extension of $T$. It is
straightforward to prove that any injective endomorphism of $S$ must
map $\{x,y\}$ to itself, and it follows that such a map must be
bijective. So $S$ is co-hopfian.
\end{example}

The counterexample for the small extensions part depends on the
following construction that builds a semigroup from a simple graph:

\begin{definition}
\label{def:sgrpfromsimpgraph}
Let $\Gamma$ be a simple graph. Let $V$ be the set of vertices of $\Gamma$. Let $S_\Gamma = V \cup \{e,n,0\}$. Define a multiplication on $S_\Gamma$ by
\begin{align*}
v_1v_2 &=
\begin{cases} 
e & \text{if there is an edge between $v_1$ and $v_2$ in $\Gamma$,} \\
n & \text{if there is no edge between $v_1$ and $v_2$ in $\Gamma$,}
\end{cases} 
&& \text{for $v_1,v_2 \in V$,} \\
ve &= ev = vn = nv = 0 && \text{for $v \in V$,} \\
en &= ne = e^2 = n^2 = 0 \\
0x &= x0 = 0 && \text{for $x \in S_\Gamma$.}
\end{align*}
Notice that all products of two elements of $S_\Gamma$ lie in
$\{e,n,0\}$ and all products of three elements are equal to $0$. Thus
this multiplication is associative and $S_\Gamma$ is a semigroup.
\end{definition}

It turns out that $S_\Gamma$ is co-hopfian (as a semigroup) if and
only if $\Gamma$ is hopfian (as a graph)
\cite[Lemma~4.5]{cm_hopf}. Furthermore, if $\Delta$ is a cofinite
subgraph of $\Gamma$, then $S_\Delta$ is a large subsemigroup of
$S_\Gamma$ \cite[Lemma~4.4]{cm_hopf}.

\begin{example}
\label{eg:cohopfreesup}
Define an undirected graph $\Gamma$ as follows. The vertex set is
\[
V = \{x_i, y_i: i \in \zset\} \cup \{z_j : j \in \nset\},
\]
and there are edges between $x_i$ and $y_i$ for all $i \in \zset$,
between $y_j$ and $z_j$ for all $j \in \nset$, and between $x_i$ and
$x_{i+1}$ for all $i \in \zset$. The graph $\Gamma$ is as shown in
\fullref{Figure}{fig:graphgamma}. Let $\Delta$ be the subgraph induced
by $W = V - \{y_0\}$; the graph $\Delta$ is as shown in
\fullref{Figure}{fig:graphgammadash}. Note that $\Gamma$ and $\Delta$
are simple.

\begin{figure}[t]
\centerline{%
\begin{tikzpicture}[node distance=5mm]
\node (x0) {$x_0$};
\node (x1) [right=of x0] {$x_1$};
\node (x2) [right=of x1] {$x_2$};
\node (x3) [right=of x2] {$x_3$};
\node (x4) [right=of x3] {};
\node (xm1) [left=of x0] {$x_{-1}$};
\node (xm2) [left=of xm1] {$x_{-2}$};
\node (xm3) [left=of xm2] {$x_{-3}$};
\node (xm4) [left=of xm3] {};
\node (y0) [above=of x0] {$y_0$};
\node (y1) [above=of x1] {$y_1$};
\node (y2) [above=of x2] {$y_2$};
\node (y3) [above=of x3] {$y_3$};
\node (ym1) [above=of xm1] {$y_{-1}$};
\node (ym2) [above=of xm2] {$y_{-2}$};
\node (ym3) [above=of xm3] {$y_{-3}$};
\node (z1) [above=of y1] {$z_1$};
\node (z2) [above=of y2] {$z_2$};
\node (z3) [above=of y3] {$z_3$};
\draw[-] (x0) edge (x1);
\draw[-] (x1) edge (x2);
\draw[-] (x2) edge (x3);
\draw[-] (x3) edge[dashed] (x4);
\draw[-] (x0) edge (xm1);
\draw[-] (xm1) edge (xm2);
\draw[-] (xm2) edge (xm3);
\draw[-] (xm3) edge[dashed] (xm4);
\draw[-] (x0) edge (y0);
\draw[-] (x1) edge (y1);
\draw[-] (x2) edge (y2);
\draw[-] (x3) edge (y3);
\draw[-] (xm1) edge (ym1);
\draw[-] (xm2) edge (ym2);
\draw[-] (xm3) edge (ym3);
\draw[-] (y1) edge (z1);
\draw[-] (y2) edge (z2);
\draw[-] (y3) edge (z3);
\end{tikzpicture}}
\caption{The graph $\Gamma$ from \fullref{Example}{eg:cohopfreesup}.}
\label{fig:graphgamma}
\end{figure}

\begin{figure}[t]
\centerline{%
\begin{tikzpicture}[node distance=5mm]
\node (x0) {$x_0$};
\node (x1) [right=of x0] {$x_1$};
\node (x2) [right=of x1] {$x_2$};
\node (x3) [right=of x2] {$x_3$};
\node (x4) [right=of x3] {};
\node (xm1) [left=of x0] {$x_{-1}$};
\node (xm2) [left=of xm1] {$x_{-2}$};
\node (xm3) [left=of xm2] {$x_{-3}$};
\node (xm4) [left=of xm3] {};
\node (y1) [above=of x1] {$y_1$};
\node (y2) [above=of x2] {$y_2$};
\node (y3) [above=of x3] {$y_3$};
\node (ym1) [above=of xm1] {$y_{-1}$};
\node (ym2) [above=of xm2] {$y_{-2}$};
\node (ym3) [above=of xm3] {$y_{-3}$};
\node (z1) [above=of y1] {$z_1$};
\node (z2) [above=of y2] {$z_2$};
\node (z3) [above=of y3] {$z_3$};
\draw[-] (x0) edge (x1);
\draw[-] (x1) edge (x2);
\draw[-] (x2) edge (x3);
\draw[-] (x3) edge[dashed] (x4);
\draw[-] (x0) edge (xm1);
\draw[-] (xm1) edge (xm2);
\draw[-] (xm2) edge (xm3);
\draw[-] (xm3) edge[dashed] (xm4);
\draw[-] (x1) edge (y1);
\draw[-] (x2) edge (y2);
\draw[-] (x3) edge (y3);
\draw[-] (xm1) edge (ym1);
\draw[-] (xm2) edge (ym2);
\draw[-] (xm3) edge (ym3);
\draw[-] (y1) edge (z1);
\draw[-] (y2) edge (z2);
\draw[-] (y3) edge (z3);
\end{tikzpicture}}
\caption{The cofinite subgraph $\Delta$ of the graph $\Gamma$ from
  \fullref{Example}{eg:cohopfreesup}.}
\label{fig:graphgammadash}
\end{figure}
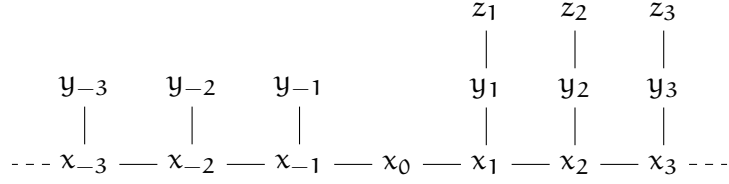

Define a map
\begin{align*}
\phi {}&: V \to V; & x_i & \mapsto x_{i+1} &&\text{for all $i \in \zset$,}\\
&& y_i & \mapsto y_{i+1} &&\text{for all $i \in \zset$,}\\
&& z_i & \mapsto z_{i+1} &&\text{for all $i \in \nset$.}
\end{align*}
It is easy to see that $\phi$ is an injective endomorphism of
$\Gamma$. However, $\phi$ is not a bijection since $z_1 \notin
\im\phi$. Thus graph $\Gamma$ is not co-hopfian. On the other hand, it
is straightforward to prove that $\Delta$ is co-hopfian.

Hence $S_\Gamma$ is non-co-hopfian small extension of the co-hopfian
semigroup $S_\Delta$.
\end{example}

As for hopficity, there is a positive result for passing to small
extensions in the finitely generated case:

\begin{theorem}[{\cite[Theorem~4.2]{maltcev_hopfian}}]
Within the class of finitely generated semigroups, co-hopficity is
inherited by small extensions but not in general by large
subsemigroups.
\end{theorem}

Note that the semigroup in \fullref{Example}{eg:cohopfreesdown} is
finitely generated, and so also function as a counterexample in the large
subsemigroups part of this result.

For small extensions, the proof proceeds by showing that if $\phi : S
\to S$ is an injective endomorphism and $X$ generates $T$, then
$X\phi^{n} \subseteq T$ for some $n \in \nset$, and so $T\phi^n
\subseteq T$ since $X$ generates $T$. Since $\phi$ is injective, so is
$\phi^n$. Since $T$ is co-hopfian, $\phi^n|_T$ is bijective, ans so
$\phi_T$ is bijective. Since $\phi_{S-T}$ is an injective and $S-T$ is
finite, $\phi_{S-T}$ is bijective. Hence $\phi$ is bijective.

\section{Subsemigroups}
\label{sec:subsemigroups}

There are a number of finiteness conditions for semigroups, based
around the notion of subsemigroups and ideals. Let us first consider
subsemigroups:

\begin{theorem}[{\cite[\S~11]{ruskuc_largesubsemigroups}}]
The property of having finitely many semigroups is inherited by small
extensions and by large subsemigroups.
\end{theorem}

\begin{theorem}[{\cite[\S~11]{ruskuc_largesubsemigroups}}]
The property of all subsemigroups being large is inherited by
extensions and by large subsemigroups.
\end{theorem}

Both of these results follow trivially from the fact that a semigroup
contains only large subsemigroups, or contains only finitely many
subsemigroups, if and only if it is finite. The interest in these
results comes from comparing them with their analogues for ideals,
discussed in the following section.

\begin{theorem}[{\cite[\S~11]{ruskuc_largesubsemigroups}}]
The property of having a minimal subsemigroup is inherited by small
extensions but not in general by large subsemigroups.
\end{theorem}

This is immediate, since a minimal subsemigroup of a semigroup $T$ is
a minimal subsemigroup of an extension of $T$. However, this property
is not inherited by large subsemigroups: Let be $T$ be a semigroup
with no minimal subsemigroups (for instance a free semigroup). Then it
is a large subsemigroup of $T^0$, which contains the minimal
subsemigroup $\{0\}$.

\section{Ideals, Green's relations, and related properties}
\label{sec:ideals}

Let us now turn to ideals, both one- and two-sided. 

\begin{theorem}[{\cite[Theorem~10.5]{ruskuc_largesubsemigroups}}]
\label{thm:suballleftidealslarge}
The property of all left ideals being large is inherited by large
subsemigroups, but not in general by small extensions.
\end{theorem}

Inheritance by large subsemigroups is proved as follows: Let $S$ be a
semigroup in which every left ideal is large and $T$ is a large
subsemigroup of $S$. Then every principal left ideal of $T$ is large
since
\[
S^1t = T^1t \cup (S-T)t.
\]
Since every left ideal is a union of principal left ideals, the result
follows.

To see that the property of all left ideals being large ideals is not
inherited by small extensions, let $T$ be an infinite semigroup in
which all left ideals are large (such as the free semigroup of rank
$1$). Then $T^0$ is a small extension of $T$, but contains the ideal
$\{0\}$, whose Rees index is infinite because $T$ is infinite.

\begin{theorem}[{\cite[Theorem~7.1]{gray_ideals} \& \cite[\S~11]{ruskuc_largesubsemigroups}}]
\label{thm:suballidealslarge}
The property of all ideals being large is inherited by large
subsemigroups but not in general by small extensions.
\end{theorem}

This is proved by reasoning analogous to
\fullref{Theorem}{thm:suballleftidealslarge}.

\begin{theorem}[{\cite[Theorem~10.3]{ruskuc_largesubsemigroups}}]
\label{thm:minleftideal}
The property of having a minimal left ideal is inherited by small
extensions, but not in general by large subsemigroups.
\end{theorem}

The proof for small extensions is by contradiction: Assume that a
semigroup $T$ has a minimal left ideal, but that a small extension $S$
of $T$ does not have one. The first step is to prove that there exist
$x,y \in T$ such that $L^T_x$ and $L^T_y$ are minimal and that $L^S_x
> L^S_y$. Repeating this, one obtains an infinite descending sequence
of $L^S_{x_1} > L^S_{x_2} > \ldots$ for which each $x_i$ lies in a
minimal $L^T$-class, or equivalently for which $T^1x_i$ is a minimal
left ideal of $T$. However, it can be shown, using finiteness of $S-T$
and the fact that $T^1x_i$ is a minimal left ideal, that $x_i = x_j$
for some $i \neq j$, contradicting the fact that $L^S_{x_i} >
L^S_{x_j}$ and completing the proof.

To see that having a minimal left ideal is not inherited by large
subsemigroups: Let $T$ be a semigroup with no minimal left ideal. Then
it is a large subsemigroup of $T^0$, which has minimal (left) ideal
$\{0\}$.

\begin{theorem}[{\cite[\S~11]{ruskuc_largesubsemigroups}}]
The property of having a minimal ideal is inherited by small
extensions, but not in general by large subsemigroups.
\end{theorem}

The reasoning proving this essentially parallels the reasoning for
\fullref{Theorem}{thm:minleftideal}.

The next property we consider is the property of having finitely many
left ideals (respectively, ideals), which is equivalent to having
finitely many $\gL$-classes (respectively, $\gJ$-classes):

\begin{theorem}[{\cite[Theorem~10.4]{ruskuc_largesubsemigroups}}]
\label{thm:finnumleftideals}
The property of having finitely many left ideals is inherited by
small extensions and by large subsemigroups.
\end{theorem}

The proof for small extensions is quite short. Let $T$ be a semigroup
with finitely left many ideals and let $S$ be a small extension of
$T$. So there are finitely many $\gL^T$-classes. Every $\gL^S$-class
is a $\gL^T$-class plus some elements of $S-T$, and so there are
finitely many $\gL^S$-classes and hence finitely many left ideals of
$S$.

The proof for large subsemigroups is longer and proceeds by
contradiction: suppose that $S$ contains finitely many left ideals and
$T$ contains infinitely many left ideals. Then some $\gL^S$-class
contains infinitely many $\gL^T$-classes. By means of a detailed
argument using products inside this $\gL^S$-class, one ultimately
derives a contradiction.

\begin{theorem}[{\cite[\S~11]{ruskuc_largesubsemigroups} \& \cite[Theorem~5.1]{gray_ideals}}]
The property of having finitely many ideals is inherited by
small extensions and by large subsemigroups.
\end{theorem}

[The property of having finitely many ideals is actually inherited by
finite Green index extensions and subsemigroups
\cite[Theorem~5.1]{gray_ideals}.]

Using $\gJ$ instead of $\gL$, one can follow the proof of the small
extensions part of \fullref{Theorem}{thm:finnumleftideals} to show
that the property of having finitely many ideals is inherited by small
extensions \cite[\S~11]{ruskuc_largesubsemigroups}.

The proof of the large subsemigroups part follows from the
corresponding result for finite Green index subsemigroups, which
basically uses Ramsey's Theorem, and does not contain many
tricks. Another use of Ramsey's Theorem, which requires some more
work, gives the following theorems about $\min_{\gL}$ and
$\min_{\gJ}$:

\begin{theorem}[{\cite[Theorems~6.1 \& 6.4]{gray_ideals}}]
The property $\min_{\gL}$ is inherited by small extensions and by
large subsemigroups.
\end{theorem}

\begin{theorem}[{\cite[Theorems~6.1 \& 6.4]{gray_ideals}}]
The property $\min_{\gJ}$ is inherited by small extensions and by
large subsemigroups.
\end{theorem}

[The properties $\min_{\gL}$ and $\min_{\gJ}$ are actually inherited
  by finite Green index extensions and subsemigroups
  \cite[Theorems~6.1 \&~6.4]{gray_ideals}.]

There are two interesting finiteness conditions whose interaction with
subsemigroups or extensions of finite Rees/Green index is less
straightforward, but which are very important in semigroup theory:
stability and the property $\gJ=\gD$. The story with $\gJ=\gD$
is particularly interesting, but let us first describe the situation
for stability.

\begin{definition}
\label{def:stable}
A $\gJ$-class $J$ of a semigroup $S$ is said to be \defterm{right
stable} if it satisfies one (and hence all) of the following
equivalent conditions \cite[Proposition 3.7]{lallement_semigroups}:
\begin{enumerate}
\item The set of all $\gR$-classes in $J$ has a minimal element.
\item There exists $q\in J$ satisfying the following property: $q\gJ
  qx$ if and only if $q\gR qx$ for all $x\in S$.
\item Every $q\in J$ satisfies the property stated in (ii).
\item Every $\gR$-class in $J$ is minimal in the set of $\gR$-classes in $J$.
\end{enumerate}
The semigroup $S$ is \defterm{right stable} if every $\gJ$-class of $S$
is right stable. The notion of left stability is defined dually. A
$\gJ$-class or a semigroup are said to be (\defterm{two-sided})
\defterm{stable} if they are both left and right stable.
\end{definition}

Roughly speaking, this definition means that if in the set of all
$\gR$-classes, contained within a fixed $\gJ$-class, there is a
minimal element, then all of these $\gR$-classes are minimal and so
pairwise incomparable. Thus, a $\gJ$-class $J$ is right stable if and
only if all the $\gR$-classes contained in $J$ are pairwise
incomparable.

It is clear that stability is a finiteness condition, and that
$\min_R$ implies right stability. Also note that stability implies the
finiteness condition $\gJ=\gD$, see~\cite{lallement_semigroups}. A
convenient way to characterize stability in algebraic term is the
following:

\begin{proposition}[{\cite[Proposition 3.10]{lallement_semigroups}}]
Let $S$ be a semigroup. Then $S$ is right stable if and only if
$R_a\leq R_{ba}$ implies $R_a=R_{ba}$ for all $a,b\in S$.
\end{proposition}

Using this characterization one can prove the following result

\begin{theorem}[{\cite[Theorem~3.2]{gray_ideals}}]
Right-, left-, and two-sided stability are inherited by small
extensions and large subsemigroups.
\end{theorem}

[Right-, left-, and two-sided stability are actually inherited
  by finite Green index extensions and subsemigroups
  \cite[Theorem~3.2]{gray_ideals}.]

The proof is basically some clever symbolic manipulation, and is not
very difficult.

The story of the property $\gJ=\gD$ is more interesting. Originally,
it was intuitively clear that the following theorem must hold, but it
took some time to establish it:

\begin{theorem}[{\cite[Theorem~4.1 \& Example~4.6]{gray_ideals}}]
\label{thm:jeqdgreen}
The property $\gJ = \gD$ is inherited by small extensions, but not in
general by large subsemigroups.
\end{theorem}

[The property $\gJ = \gD$ is actually inherited by finite Green index
  extensions \cite[Theorem~4.1]{gray_ideals}.]

The small extensions part of the result follows from
\fullref{Example}{ex:jseqjdjtnewdt} below. This is not the first
counterexample exhibited \cite[Example~6.4]{gray_ideals}. Rather, this
is an improved example that appears here for the first time, showing
that the property $\gJ = \gD$ is not inherited by a subsemigroup whose
complement contains only a single element. Only the statement of the
example is given here; the proof is long and so is given in
\fullref{Proposition}{prop:jseqjdjtnewdt} in the Appendix.

\begin{example}
\label{ex:jseqjdjtnewdt}
Let $S$ be the semigroup presented by:
\begin{align*}
S=\mathrm{Sg}\bigl\langle\, a,b,c,d,e,f,x,y :{}& cabd=a, eca=a, abf^2=ab,\\
& abf=xa, yabf=a, f^3=f,\\
& ec=ce, ex=xe, ey=ye,\\
& cx=xc, cy=yc, xy=yx \bigr\rangle.
\end{align*}
Then $T=S - \{f\}$ is a large subsemigroup of $S$, and $\gJ=\gD$ in
$S$ but $\gJ\neq\gD$ in $T$ by
\fullref{Proposition}{prop:jseqjdjtnewdt}.
\end{example}

The story of the property $\gJ=\gD$ is not finished! If at least one
of the subsemigroup and the original semigroup is regular, $\gJ=\gD$
\emph{is} inherited by large subsemigroups:

\begin{theorem}[{\cite[Theorems~4.7 \& 4.8]{gray_ideals}}]
Let $S$ be a semigroup and let $T$ be a large subsemigroup of
$S$. Suppose at least one of $S$ and $T$ is regular. Then if $\gJ=\gD$
in $S$, then $\gJ=\gD$ in $T$.
\end{theorem}


\section{Automata and semigroups}
\label{sec:automata}

\subsection{Word-hyperbolicity}
\label{sec:wordhyp}

Hyperbolic groups --- groups whose Cayley graphs are
hy\-per\-bol\-ic metric spaces --- have grown into one of the most
fruitful areas of group theory since the publication of Gromov's
seminal paper \cite{gromov_hyperbolic}. Duncan \& Gilman
\cite{duncan_hyperbolic} generalized the concept of hyperbolicity to
semigroups and monoids as follows: a semigroup is word-hyperbolic if
there is a pair $(L,M(L))$, where $L$ is a regular language over an
alphabet $A$ representing a finite generating set for $S$, such that
$L$ maps onto $S$ and the language
\[
M(L) = \{u\#_1v\#_2w^\rev : u,v,w \in L \land uv =_S w\}
\]
(where $\#_1$ and $\#_2$ are new symbols not in $A$ and $w^\rev$
denotes the reverse of the word $w$) is con\-text-free. In this case,
the pair $(L,M(L))$ is a \defterm{word-hy\-per\-bol\-ic structure}
for $S$. A semigroup is \defterm{word-hy\-per\-bol\-ic} if it admits a
word-hy\-per\-bol\-ic structure. For groups, this definition is
equivalent to the usual notion of hyperbolicity
\cite{gilman_hyperbolic}.

\begin{theorem}
\label{thm:wordhyplarge}
Word-hyperbolicity is inherited by large subsemigroups.
\end{theorem}

This result is an immediate consequence of the corresponding result for
finite Green index subsemigroups:

\begin{theorem}
\label{thm:wordhypgreensub}
Word-hyperbolicity is inherited by finite Green index subsemigroups.
\end{theorem}

Both \fullref{Theorem}{thm:wordhyplarge} and
\fullref{Theorem}{thm:wordhypgreensub} appear here for the first time.

\begin{proof}
Let $S$ be a semigroup admitting a word-hyperbolic structure
$(L,M(L))$, where $L$ is over an alphabet $A$ representing a finite
generating set for $S$. Let $T$ be a subsemigroups of $S$ of finite
Green index. As per \cite[\S10]{cgr_greenindex}, there is a synchronous
rational relation $R \subseteq A^+ \times B^+$, where $B$ is a
particular alphabet representing a generating set for $T$, such that
\begin{enumerate}
\item If $p \in A^+$ represents an element of $T$, then there is a
  unique string $p' ∈ B^+$ with $(p,p') \in R$ and $\elt{p} = \elt{p'}$.
\item If $(p, p') \in R$ then $|p| = |p'|$, and $\elt{p} = \elt{p'}$ and
  so $p$ represents an element of $T$.
\end{enumerate}

Since $R$ is a synchronous rational language, so is
\begin{align*}
R' &= R(\#_1,\#_1)R(\#_2,\#_2)R^\rev \\
&= \{(u\#_1v\#_2w^\rev,u'\#_1v'\#_2w'^\rev) : (u,u') \in R, (v,v') \in R, (w,w') \in R\}.
\end{align*}

Let $K = L \circ R = \{p' \in B^+ : (\exists p\in L)((p,p') \in
R)\}$. Notice that $K$ is regular. Then
\begin{align*}
M(K) &= \{u'\#_1v'\#_2w'^\rev : u',v',w' \in K \land {u'}\;{v'}=_T{w'}\} \\
&= \{u'\#_1v'\#_2w'^\rev : (\exists u,v,w \in L)((u,u'),(v,v'),(w,w') \in R) \land {u'}\;{v'}=_T{w'}\} \\
&= \{u'\#_1v'\#_2w'^\rev : (\exists u,v,w \in L)((u,u'),(v,v'),(w,w') \in R) \land {u}\;{v}=_T{w}\} \\
&\qquad\qquad \text{(by property (ii) of $R$)} \\
&= \{u'\#_1v'\#_2w'^\rev : (\exists u,v,w \in L)((u\#v\#w^\rev,u'\#v'\#w'^\rev) \in R') \land {u}\;{v}=_T{w}\} \\
&= M(L) \circ R'.
\end{align*}
Therefore, since $M(L)$ is a context-free language, so is
$M(K)$. Therefore $(K,M(K))$ is a hyperbolic structure for $T$.
\end{proof}

\begin{question}
Is word-hyperbolicity inherited by small extensions?
\end{question}

It is an open question whether adjoining an identity to a
word-hyperbolic semigroups necessarily yields a word-hyperbolic monoid
\cite[Question~1]{duncan_hyperbolic}. Duncan \& Gilman pointed out
that this question would have an affirmative answer if word-hyperbolic
semigroups were necessarily to admit word-hyperbolic structures with
uniqueness (where every element of the semigroup has a unique
representative in the regular language)
\cite[Question~2]{duncan_hyperbolic}, but it is now known that this
does not hold~\cite[Example~4.2]{cm_wordhypunique}.

\subsection{Automaticity and asynchronous automaticity}
\label{sec:auto}

Automaticity for groups \cite{epstein_wordproc} has been generalized
to semigroups \cite{campbell_autsg}. Let $A$ be a finite alphabet
representing a generating set for a semigroup $S$. Let $L$ be a
regular language over $A$ such that every element of $S$ has at least
one representative in $L$. For any $w \in A^*$, define the relation
\[
L_w = \{(u,v) : u,v \in L, \elt{uw} = \elt{v}\}.
\]
The pair $(A,L)$ forms an \defterm{automatic structure} for $S$ if the
language $L_a$ is a synchronous rational relation (in the sense of
being recognized by a synchronous two-tape finite automaton) for each
$a \in A \cup \{\emptyword\}$. The pair $(A,L)$ forms an
\defterm{asynchronous automatic structure} for $S$ if the language
$L_a$ is a rational relation (in the sense of being recognized by a
possibly asynchronous two-tape finite automaton) for each $a \in A
\cup \{\emptyword\}$ (see \cite[Definition 3.3]{hoffmann_relatives}
for details). An \defterm{automatic semigroup} (respectively,
\defterm{asynchronous automatic semigroup}) is a semigroup that admits
an automatic (respectively, asynchronous automatic) structure.

\begin{theorem}[{\cite[Theorem~1.1]{hoffmann_autofinrees}}]
\label{thm:auto}
Automaticity is inherited by small extensions and by large subsemigroups.
\end{theorem}

To prove the small extensions part of the result, proceed as follows:
Let $T$ be a semigroup that admits an automatic structure $(A,L)$ and
let $S$ be a small extension of $T$. Let $C$ be a finite set of
symbols in bijection with $S-T$. Let $A' = A \cup C$ and $L' = L \cup
C$. It can be shown that $(A', L')$ is an automatic structure for
$S$. Proving that the various relations $L'_a$ for $a \in A' \cup
\{\emptyword\}$ are regular involves first constructing some auxiliary
regular relations that describe how elements represented by letters in
$C$ multiply elements represented by words in $A^*$.

The inheritance of automaticity by large subsemigroups is more easily
deduced as a corollary of inheritance by finite Green index
subsemigroups:

\begin{theorem}[{\cite[Theorems~10.1 \&~10.2]{cgr_greenindex}}]
\label{thm:autogreensub}
Automaticity and asynchronous automaticity are inherited by finite
Green index subsemigroups.
\end{theorem}

The proof of this result uses a similar strategy to that described in
the proof of \fullref{Theorem}{thm:wordhypgreensub} above. Let $S$ be a
semigroup with an automatic (respectively, asynchronous automatic)
structure $(A,L)$. Let $T$ be a finite Green index subsemigroup of
$S$. Let $R \subseteq A^+ \times B^+$ be as in the proof of
\fullref{Theorem}{thm:wordhypgreensub}. Let $K = L \circ R$. Then $K$ is
regular and $K_b = R^{-1} \circ L_w \circ R$, where $w \in A^*$ is
some word representing the same element as $b$. Since $L_w$ is a
synchronous (respectively, asynchronous) rational relation, so is
$K_b$. Hence $(B,K)$ is an automatic (respectively, asynchronous
automatic) structure for $T$.

[Even written out in full, the above proof is rather shorter and
  simpler than the original proof for large subsemigroups
  \cite[\S~4]{hoffmann_autofinrees}.]

\subsection{Markovicity}
\label{sec:markov}

A semigroup is \defterm{Markov} if it admits a regular language of
unique normal forms that is closed under taking non-empty prefixes,
known as a \defterm{Markov language}. The concepts of Markov groups was
introduced by Gromov in his seminal paper on hyperbolic groups
\cite[\S~5.2]{gromov_hyperbolic}, and has recently been extended to
semigroups and monoids \cite{cm_markov}.

\begin{theorem}[{\cite[Theorem~16.1]{cm_markov}}]
Markovicity is inherited by small extensions and by large
subsemigroups.
\end{theorem}

Let $S$ be a semigroup and $T$ a large subsemigroup of $S$. If $T$ is
Markov, and $L$ is a Markov language for $T$, then $L \cup (S-T)$ is a
Markov language for $S$. Thus Markovicity is inherited by small
extensions. In the other direction, the proof is more complex: in
outline, one can show that the rewriting map $\phi$ from
\fullref{\S}{sec:fingen} can be computed by a transducer, and then one
can apply this to a Markov language for $S$ to obtain a Markov
language for $T$.

\subsection{Automatic presentations}
\label{sec:fap}

Automatic presentations, also known as FA-pre\-sent\-a\-tions, were
introduced by Khoussainov \& Nerode \cite{khoussainov_autopres} to
fulfill a need to extend finite model theory to infinite structures
while retaining the solubility of interesting decision problems, and
have recently been applied to algebraic structures such as groups
\cite{oliver_autopresgroups}, rings \cite{nies_rings}, and
semigroups \cite{cort_apcancsg,cort_apsg}.

\begin{definition}
Let $S$ be a semigroup. Let $L$ be a regular language over a finite
alphabet $A$, and let $\phi : L\rightarrow S$ be a surjective
mapping. Then $(L,\phi)$ is an automatic presentation for $S$ if the
relations
\[
\Lambda(=,\phi) = \{(w_1,w_2)\in L^2:w_1\phi=w_2\phi\}
\]
and
\[
\Lambda(\circ,\phi) = \{(w_1,w_2,w_3)\in L^3: (w_1\phi)(w_2\phi) = w_3\phi\}
\]
are regular, in the sense of being recognized by multi-tape
synchronous finite automata.

If $(L,\phi)$ is an automatic presentation for $S$ and the
mapping $\phi$ is injective (so that every element of the semigroup
has exactly one representative in $L$), then $(L,\phi)$ is said to be
\defterm{injective}.

If $(L,\phi)$ is an automatic presentation for $S$ and $L$ is
a language over a one-letter alphabet, then $(L,\phi)$ is a
\defterm{unary} automatic presentation for ${S}$, and
${S}$ is said to be \defterm{unary FA-pre\-sent\-a\-ble}.
\end{definition}

\begin{theorem}
Admitting an automatic presentation is inherited by large
subsemigroups, but not in general by small extensions.
\end{theorem}

\begin{theorem}
Admitting a unary automatic presentation is inherited by large
subsemigroups, but not in general by small extensions.
\end{theorem}

To prove the large subsemigroups part of the result, proceed as
follows. Let $S$ be a semigroup admitting an automatic presentation
(respectively, unary automatic presentation) $(L,\phi)$ and let $T$ be
a large subsemigroup of $S$. Assume without loss that $\phi$ is
injective \cite[Theorem~3.4]{blumensath_diploma}. Let $K =
(S-T)\phi^{-1}$. Since $S-T$ is finite and $\phi$ is injective, $K$ is
a finite subset of $L$ and therefore regular. So $L' = L - K$ is
regular, and $L'\phi|_{L'} = T$. Finally,
\begin{align*}
\Lambda(=,\phi|_{L'}) &= \Lambda(=,\phi) \cap (L' \times L'), \\
\Lambda(\circ,\phi|_{L'}) &= \Lambda(\circ,\phi) \cap (L' \times L' \times L'),
\end{align*}
and so $(L',\phi|_{L'})$ is an automatic presentation (respectively,
unary automatic presentation) for $T$.

The following counterexample shows that a small extension of a
semigroup admitting a unary automatic presentation is does not in
general admit an automatic presentation.

\begin{example}
\label{ex:finreesup}
Define a semilattice $S$ as follows. The set of elements is $\{s_i,t_i
: i \in \nset \cup \{0\}\}$, and the order $\leq$ is defined on $S$ as
follows: for all $i,j \in \nset$,
\begin{align*}
t_i \leq t_j &\iff i \leq j \\
t_i \leq s_j &\iff i \leq j \\
s_i \leq s_j &\iff i = j \\
s_i \not\leq t_j.
\end{align*}
The Hasse diagram for $(S,\leq)$ is as illustrated in
\fullref{Figure}{fig:fap}(a).

Let $Y \subseteq \nset \cup \{0\}$ be non-recursively enumerable. Let
$U = S \cup \{e\}$ and extend the relation $\leq$ to $U$ as
follows: for $i \in \nset$, 
\begin{align*}
t_i \leq e \\
s_i \leq e &\iff i \in Y.
\end{align*}
The Hasse diagram for $(U,\leq)$ is as illustrated in
\fullref{Figure}{fig:fap}(b).

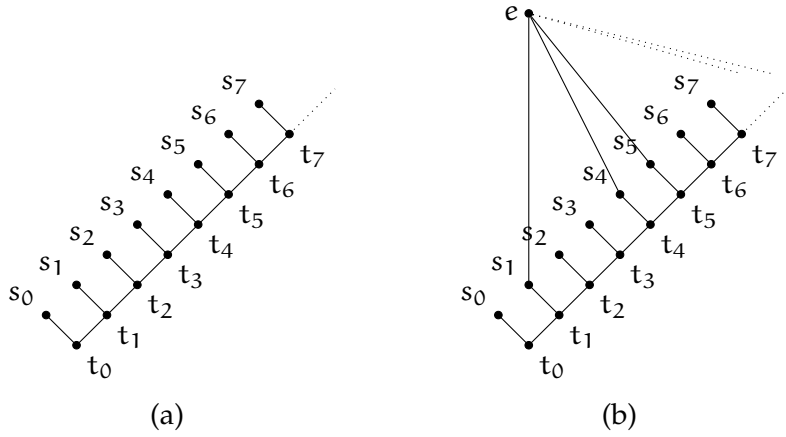
\begin{figure}[t]
\centerline{%
\begin{tabular}{c@{\hspace{15mm}}c}
\begin{tikzpicture}[x=.4cm,y=.4cm,vertex/.style={circle,draw,fill=black,inner sep=1pt}],
\foreach\x in {0,1,...,7} {
  \node[vertex] at (\x,\x) {};
  \node[anchor=north west] at (\x,\x) {$t_{\x}$};
  \node[vertex] at (\x-1,\x+1) {};
  \node[anchor=south east] at (\x-1,\x+1) {$s_{\x}$};
  \draw (\x-1,\x+1)--(\x,\x);
}
\draw (0,0)--(7,7);
\draw[dotted] (7,7)--(8.5,8.5);
\end{tikzpicture}
&
\begin{tikzpicture}[x=.4cm,y=.4cm,vertex/.style={circle,draw,fill=black,inner sep=1pt}],
\foreach\x in {0,1,...,7} {
  \node[vertex] at (\x,\x) {};
  \node[anchor=north west] at (\x,\x) {$t_{\x}$};
  \node[vertex] at (\x-1,\x+1) {};
  \node[anchor=south east] at (\x-1,\x+1) {$s_{\x}$};
  \draw (\x-1,\x+1)--(\x,\x);
}
\draw (0,0)--(7,7);
\draw[dotted] (7,7)--(8.5,8.5);
\node[vertex] (e) at (0,11) {};
\node[anchor=east] at (e) {$e$};
\draw (e)--(0,2);
\draw (e)--(3,5);
\draw (e)--(4,6);
\draw[dotted] (e)--(7,9);
\draw[dotted] (e)--(8,9);
\end{tikzpicture}
\\
(a) & (b) 
\end{tabular}}
\caption{Hasse diagrams for (a) $(S,\leq)$ and (b) $(U,\leq)$, assuming for the sake of illustration that $1$, $4$, $5$ lie in $Y$.}
\label{fig:fap}
\end{figure}

The semilattices $(S,\leq)$ and $(U,\leq)$ can be viewed as meet
semigroups $(S,\land)$ and $(U,\land)$.

Let $\phi : a^* \to S$ be defined by $a^{2i} \mapsto s_i$ and
$a^{2i+1} \mapsto t_i$ for all $i \geq 0$. Then it is easy to see that
$\Lambda(=,\phi)$ is the diagonal relation and that
$\Lambda(\land,\phi)$ is also regular. So $(a^*,\phi)$ is a unary
automatic presentation for the semigroup $(S,\land)$.

However, the semigroup $(U,\land)$ does not admit an automatic
presentation: using the fact that a structure admitting an automatic
presentation has solvable first-order theory
\cite[Corollary~4.2]{khoussainov_autopres}, it is straightforward to
prove that if it did admit an automatic presentation, there would be
an algorithm to enumerate $Y$, contradicting the fact that $Y$ is not
recursively enumerable.
\end{example}

\section{Geometry}
\label{sec:geometry}

\subsection{Hyperbolicity}
\label{sec:hyperbolic}

\begin{definition}
A geodesic space $(X,d)$ is $\delta$-hyperbolic if, for every three
points $x,y,z$ and geodesics $\alpha$, $\beta$, $\gamma$ from $x$ to
$y$, $y$ to $z$, and $z$ to $x$ respectively, then for every point $u$
on $\alpha$, the distance from $u$ to $\beta \cup \gamma$ is less than
$\delta$. (If this holds, then by interchanging $x$, $y$, and $z$,
appropriately, one sees that every point on $\beta$ is within $\delta$
of $\alpha \cup \gamma$, and every point on $\gamma$ is within
$\delta$ of $\alpha \cup \beta$.)
\end{definition}

\begin{definition}
Let $S$ be a monoid generated by a set $A$. Define a metric $d$ on
the Cayley graph $\Gamma(S,A)$ by defining $d(u,v)$ to be the length
of the shortest undirected path connecting $u$ and $v$. Under this
definition, $\Gamma(S,A)$ is technically not a geodesic space, but it
can be made into one by extending the metric $s$ to the whole of the
Cayley graph by making each edge isometric to the interval $[0,1]$. A
monoid is \defterm{hyperbolic} if its Cayley graph (with respect to
some generating set) is hyperbolic.
\end{definition}

The definition of hyperbolicity is limited to monoids because the
Cayley graph of a semigroup without an identity is not necessarily
connected.

\begin{theorem}[Folklore]
Within the class of monoids, hyperbolicity is not inherited by large
subsemigroups.
\end{theorem}

To see this, let $T$ be any non-hyperbolic monoid (for example, $\zset \times
\zset$) and let $S = \adjzero{T}$. Then every pair of elements in $S$
is a bounded distance apart via a path running through the zero. Hence
$S$ is trivially hyperbolic but contains a non-hyperbolic large
subsemigroup $T$.

\begin{question}
Within the class of monoids, is hyperbolicity inherited by small extensions?
\end{question}

The analogous graph-theoretical question has a negative answer: it is
easy to exhibit a non-hyperbolic graph with a hyperbolic subgraph
where the complement contains a single vertex. So if this question has
a positive answer, it is somehow dependent on the restricted nature of
Cayley graphs of monoids.

However, positive results can be proved within the class of monoids
with the following property:

\begin{definition}
Let $S$ be a monoid generated by a finite set $A$. The monoid $S$ is
of \defterm{finite geometric type} (\fgt) if there is a constant $n$
such that, for every $q \in S$ and $a \in A$, there are at most $n$
distinct solutions $x$ to the equation $xa = q$.
\end{definition}

\begin{theorem}
\label{thm:hyperbolicfgt}
Within the class of monoids of finite geometric type, hyperbolicity
is inherited by small extensions and by large subsemigroups.
\end{theorem}

This follows immediately from \fullref{Theorem}{thm:reesquasi} below
and the fact that hyperbolicity is preserved under quasi-isometries
\cite[Theorem~5.12]{ghys_hyperbolic}. Recall the definition of a
quasi-isometry: a map $\phi$ from a metric space $(X,d_X)$ to another
metric space $(Y,d_Y)$ is a \defterm{$(k,\epsilon,c)$-quasi-isometry},
where $k,\epsilon,c \in \rset$, if
\[
(\forall x,x' \in X)\Bigl(\frac{1}{k}d_X(x,x') - \epsilon \leq d_Y(x\phi,x'\phi) \leq kd_X(x,x') + \epsilon\Bigr),
\]
and 
\[
(\forall y \in Y)(\exists x \in X)(d_Y(y,x\phi) \leq c).
\]

\begin{theorem}
\label{thm:reesquasi}
Let $S$ be a monoid of finite geometric type and let $T$
be a finite Rees index submonoid of $S$. Then the natural
embedding map $T \hookrightarrow S$ is a quasi-isometry.
\end{theorem}

\begin{proof}
Let $A$ be a finite generating set for $T$. Then $A \cup (S-T)$
generates $S$ and the Cayley graph $\Gamma_S = \Gamma(S,A\cup(S-T))$
contains the Cayley graph $\Gamma_T = \Gamma(T,A)$ as a
subgraph. Denote the distance function in $\Gamma_S$ by $d_S$ and the
distance function in $\Gamma_T$ by $d_T$.

Consider arbitrary elements $t_1$ and $t_2$ of $T$.  Notice first that
$d_S(t_1,t_2) \leq d_T(t_1,t_2)$, since $\Gamma_T$ is a subgraph of
$\Gamma_S$.  The aim is to find a constant $k$, dependent only on $S$,
$T$, and $A$, such that $d_T(t_1,t_2) \leq k d_S(t_1,t_2)$. It will
then follows that the embedding map is a $(k,0)$-quasi-isometry.

Let
\[
B = \{t \in T : (\exists a \in A \cup (S-T), s \in S-T)((ta = s) \lor (sa = t)\}.
\]
In the definition of $B$, there are only finitely many choices for $a$
and $s$, and thus, since $S$ is of finite geometric type, there are
only finitely many possibilities for $t$; thus $B$ is finite. Let
\[
k_1 = \max\{d_T(t,t') : t,t' \in B\}.
\]
Let $f$ be such that there are at most $f$ coterminal edges with the
same label in $\Gamma_S$. Let $g = f|S-T|$. Let
\[
k_2 = \max\{d_T(t,ts) : t,ts \in T, s \in (S-T), |t|_T \leq g\}.
\]
Let $h = \max\{|as|_T : a \in A, s \in S-T, as \in T\}$. Let
\[
k_3 = h + g.
\]
Let $k = \max\{k_1,k_2,k_3\}$. Let $c=|S-T|+1$. The aim is to show
that the embedding map is a $(k,0,2)$-quasi-isometry. 

To this end, let $t_1,t_2 \in T$ be arbitrary. Let $\gamma$ be an
$\Gamma_S$-geodesic from $t_1$ to $t_2$. Let us construct a
path from $t_1$ to $t_2$ that lies entirely in $\Gamma_T$ and whose
length is at most a $k$ times that of $\gamma$.

\begin{figure}[tb]
\centerline{%
\begin{tikzpicture}[x=2cm,y=2cm,vertex/.style={circle,draw,fill=black,inner sep=1pt}],
\node[vertex] (tone) at (0,0) {};
\node[vertex] (ttwo) at (4,.5) {};
\node[vertex] (p) at (1.5,1.3) {};
\node[vertex] (q) at (3,1.2) {};
\coordinate (phalf) at ($ (tone)!.5!(p) $);
\coordinate (qhalf) at ($ (q)!.5!(ttwo) $);
\coordinate (selt) at ($ (p)!.5!(q) + (0,.4) $);
\coordinate (telt) at ($ (p)!.5!(q) - (0,.3) $);
\coordinate (leftbound) at (-.2,1.4);
\coordinate (rightbound) at (4.2,1.4);
\draw[dashed] (tone)--(p);  
\draw (phalf)--(p);  
\draw (q)--(qhalf);
\draw[dashed] (qhalf)--(ttwo);
\draw[dotted] (leftbound)--(rightbound);
\path (p) edge[bend left=50] node[anchor=south west] {$\beta_i$} (q);
\path[dashed] (p) edge[bend right=40] (q);
\node[anchor=south west] at (leftbound) {$S-T$};
\node[anchor=north west] at (leftbound) {$T$};
\node[anchor=east] at (p) {$p$};
\node[anchor=west] at (q) {$q$};
\node[anchor=east] at (tone) {$t_1$};
\node[anchor=west] at (ttwo) {$t_2$};
\end{tikzpicture}}
\caption{Replacing a subpath $\beta_i$ with a subpath of length at most $k_1$ lying entirely in $\Gamma_T$.}
\end{figure}
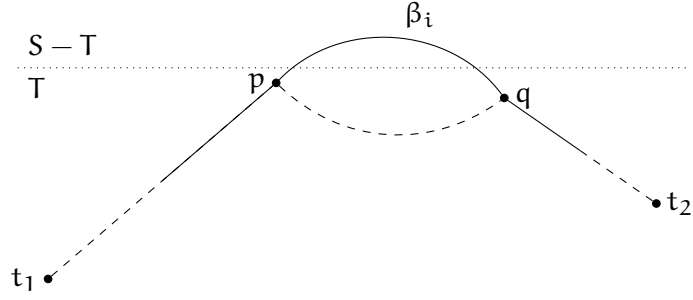
Suppose first that $\gamma$ visits some vertices in $S-T$. Then
$\gamma = \alpha_0\beta_1\alpha_1\cdots \beta_n\alpha_n$, where every
$\alpha_i$ contains only vertices from $T$ and every $\beta_i$
contains only vertices from $S-T$, and every $\alpha_i$, $\beta_i$
contains at least one vertex. Suppose the last vertex in
$\alpha_{i-1}$ is $p \in T$ and the first vertex $\alpha_i$ is $q \in
T$. Notice that $d_S(p,q) \geq 2$. Now, $p, q \in B$ and so $d_T(p,q)
\leq k_1$. So $\beta_i$ can be replaced by a subpath of length at most
$k_1$ lying entirely in $\Gamma_T$. Doing this for all $\beta_i$ yields a new
path $\gamma'$ that only visits vertices in $T$. Notice that
$|\gamma'| \leq k_1|\gamma|$.

This new path $\gamma'$ may, however, still contain edges labelled by
elements of $S-T$. These can only lie on the subpaths $\alpha_i$,
since the subpaths that replaced the $\beta_i$ lie entirely in
$\Gamma_T$. Consider such an edge from $p$ to $q$ labelled by $s \in
S-T$. If $|p|_T \leq g$, this edge can be replaced by a path in
$\Gamma_T$ of length $k_2$.

Otherwise let $p = a_1\cdots a_n$, where $a_i \in A$ and
$n$ is minimal with $n > g = f|S-T|$. Consider the elements
\begin{equation}
\label{eq:finrees1}
a_ns, a_{n-1}a_ns, \ldots, a_{n-g}\cdots a_ns.
\end{equation}
Now, each of the $g+1$ products $a_n, a_{n-1}a_n, \ldots,
a_{n-g}\cdots a_n$ is distinct, for otherwise $n$ would not be
minimal. Suppose that all the elements \eqref{eq:finrees1} lie in
$S-T$. Then, since there are $f|S-T|+1$ of them, at least $f+1$ of
them must equal the same element $x$ of $S-T$. But then there are
$f+1$ edges coterminal at $x$ with label $s$. This is a contradiction,
and so some one of the elements \eqref{eq:finrees1} lies in $T$. Let
$i$ be minimal such that $a_{n-i}\cdots a_ns$ lies in $T$. Then
$a_{n-i+1}\cdots a_ns = s' \in S-T$ and $a_{n-i}s' \in T$. So there is
a word $u \in A^+$ of length at most $h$ with $a_{n-i}s' = u$. Since
the distance from $p$ to $a_1\cdots a_{n-i-1}$ is at most $g$, there
is a path from $p$ to $q$, entirely in $\Gamma_T$, of length at most
$k_3 = g+h$.

So, any such edge in $\gamma'$ can be replaced by a path in $\Gamma_T$
of length at most $\max\{k_2,k_3\}$. Replacing every edge in this way,
we obtain a path $\gamma''$, entirely in $\Gamma_T$, of length at most
$k|\gamma|$.

Therefore $d_S(t_1,t_2) \leq d_T(t_1,t_2) \leq kd_S(t_1,t_2)$ for any
$t_1,t_2 \in T$.

Finally, note that every point in $s - S-T$ is a distance at most $2$
from a point in $T$, via the edge labelled by $s$ from the identity of
the monoid $S$ to $s$ and any edge labelled by $a \in A$.

Consequently, the embedding map $T \hookrightarrow S$
is a $(k,0,2)$-quasi-isometry.
\end{proof}

A consequence of the celebrated Rips construction
\cite{rips_subgroups} is that hyperbolic groups can contain
non-hyperbolic (indeed non-finitely presented) subgroups. Thus
\fullref{Theorem}{thm:hyperbolicfgt} result does not hold for
arbitrary subsemigroups.

\subsection{Ends}
\label{sec:ends}

One can define the notion of ends of finitely generated semigroup: for
a finite generating set $A$ for a semigroup S, one considers the
underlying undirected graph of $\Gamma(S,A)$ and considers its number
of ends, see~\cite{kilibarda_ends}. One indeed proves that the number
of ends does not depend on the finite generating set. The following
theorem holds:

\begin{theorem}
Let $S$ be a finitely generated semigroup and let $T$ be a large
subsemigroup of $S$. Then the number of ends of $S$ coincides with the
number of ends of $T$.
\end{theorem}

The proof of this theorem involves some rewriting procedure similar to
the proof of the previous theorem. Strangely, for finite Green index
this theorem is no longer true, but with additional condition of
can\-cel\-la\-tiv\-ity of $S$, it generalizes to Green index, too.

\section{Bergman's property \& cofinality}
\label{sec:bergman}

A semigroup $S$ is said to have \emph{Bergman's property} if for any
generating set $A$ for $S$ there exists $n=n(A)\geq 1$ such that
$S=A\cup\cdots\cup A^n$. It is so-called after Bergman
\cite{bergman_generating} noticed that infinite symmetric groups
satisfy this property. In \cite[Proposition~2.2]{maltcev_bergman} the
authors initiated the study of Bergman's property for semigroups. It
turns out that the following two natural versions of algebraic
cofinality are very closely related to the Bergman's property.

\begin{definition}
Let $S$ be a non-finitely generated semigroup. The \emph{cofinality}
$\cf(S)$ of $S$ is the least cardinal $\kappa$ with the
property that there exists a chain of proper subsemigroups
$(S_i)_{i<\kappa}$ such that $\bigcup_{i<\kappa}S_i=S$.

Let $S$ be a non-finitely generated semigroup. The \emph{strong
  cofinality} $\scf(S)$ of $S$ is the least cardinal $\kappa$
with the property that there exists a chain of proper subsets
$(S_i)_{i<\kappa}$ such that $\bigcup_{i<\kappa}S_i=S$, and
$S_iS_i\subseteq S_{i+1}$ for all $i<\kappa$.
\end{definition}

The following result gives all the information needed about the
relationship of Bergman's property, cofinality, and strong cofinality:

\begin{proposition}[{\cite[Proposition~2.2]{maltcev_bergman}}]
Let $S$ be a non-finitely generated semigroup. Then
\begin{enumerate}
\item $\scf(S)>\aleph_0$ if and only if
  $\cf(S)>\aleph_0$ and $S$ has Bergman's property;
\item If $\scf(S)>\aleph_0$, then
  $\scf(S)=\cf(S)$.
\end{enumerate}
\end{proposition}

Essentially, this proposition says that there are four different cases
with respect to having/not having Bergman's property and admitting/not
admitting countable cofinality. All four cases do in fact arise; see
\cite[\S~2]{maltcev_bergman}.

The fact that semigroups with uncountable strong cofinality always
have Bergman's property gives a very convenient way to construct
examples of semigroups with Bergman's property: one can use various
techniques to show that a semigroup has uncountable strong cofinality,
including the diagonal argument. But what is really convenient, is to
use the following --- quite surprising --- result, which characterizes
semigroups with uncountable strong cofinality in terms of length
functions:

\begin{proposition}[{\cite[Lemma~2.3]{maltcev_bergman}}]
Let $S$ be a non-finitely generated semigroup. Then
$\scf(S)>\aleph_0$ if and only if every function
$\Phi:S\to\nset$ such that there exists $k=k(\Phi)\geq 1$ such that
\begin{equation*}
(st)\Phi\leq (s)\Phi+(t)\Phi+k\quad\text{for all $s,t\in S$}
\end{equation*}
is bounded above.
\end{proposition}

This characterization is the key to proving the following result:

\begin{theorem}[{\cite[Theorem~3.2]{maltcev_bergman}}]
Let $S$ be a non-finitely generated semigroup, and $T$ be a large
subsemigroup in $S$. Then $\cf(S)=\cf(T)$ and $\scf(S)=\scf(T)$.
\end{theorem}

\begin{theorem}[{\cite[Theorem~3.2]{maltcev_bergman}}]
Within the class of non-finitely generated semigroups, Bergman's
property is inherited by small extensions.
\end{theorem}

\begin{question}
Within the class of non-finitely generated semigroups, is Bergman's
property inherited by large subsemigroups?
\end{question}

\section{Appendix}

\begin{proposition}
\label{prop:jseqjdjtnewdt}
In \fullref{Example}{ex:jseqjdjtnewdt}, $\gJ^S=\gD^S$ but
$\gJ^T\neq\gD^T$.
\end{proposition}

\begin{proof}
Observe the following relations in $S$:
\begin{align*}
abd &= ecabd=ea\\
xaf &= abf^2=ab\\
yab &= yabf^2=af\\
af^2 &= yabf^3=yabf=a\\
afd &= yabf^2d=yabd=yea.
\end{align*}
Furthermore, $xya=yxa=yabf=a$. Hence, recalling that $eca=cea=a$, if
for $w\in\{e,c,x,y\}^{\ast}$ we denote by $w'$ the word
$e^{k_e}c^{k_c}y^{k_y}x^{k_x}$, where
\begin{align*}
k_e &= |w|_e-\min(|w|_e,|w|_c)\\
k_c &= |w|_c-\min(|w|_e,|w|_c)\\
k_y &= |w|_y-\min(|w|_x,|w|_y)\\
k_x &= |w|_x-\min(|w|_x,|w|_y),
\end{align*}
then $wa=w'a$. It is routine to check that the rewriting system
\begin{multline*}
\bigl\{abd\to ea, f^3\to f, abf\to xa, xaf\to ab, yab\to af, af^2\to
a, afd\to yea,\\
xy\to yx, ce\to ec, xe\to ex, xc\to cx, ye\to ey, yc\to cy,\\
wa\to w'a\quad\mbox{for all} w\in\{e,c,x,y\}^{\ast}\bigr\}
\end{multline*}
is confluent. It is obvious that this system is terminating. Hence
we can work with the normal forms for the elements of $S$.

It follows from the presentation that if $f=uv$ for some $u,v\in S$,
then $u=f^k$ and $v=f^n$ for some $k,n\geq 1$. Then both $k$ and $n$
cannot be even, as otherwise $f=uv=f^2$. Therefore one of $k$ and
$n$ is odd and so $u=f$ or $v=f$. Therefore $T=S\setminus\{f\}$ is
indeed a subsemigroup of $S$.

First we prove that $\gJ^T\neq\gD^T$. Since
$ab=1\cdot a\cdot b$ and $a=c\cdot ab\cdot d$, we obtain that
$a\gJ^T ab$. Assume that $a\gD^T ab$. Then there
exists $m\in T$ such that $a\gL^T m\gR^T ab$. Then
$m=wa$ for some $w\in T^1$. There exists $u$ such that $uwa=a$.
Since the relations preserve the number of $a$'s in the words, it
follows that $w$ can contain only letters $e,c,x,y,b,f,d$. From the
rewriting rules and $uwa=a$ it follows that $uw$ cannot contain
letters $b,f,d$. Hence $w\in\{e,c,x,y\}^{\ast}$. Now, from
$wa\gR^T ab$ it follows that there exists $p,q\in T^1$ such
that $abp=wa$ and $waq=ab$. We can also assume that $p$ is in its
normal form. The normal form for $wa$ is $w'a$. So, $abp$ must be
not in its normal form and so, by inspection we deduce that $p$ must
start either with $d$, or with $f$. If $p=dp_1$, then
$abp=abdp_1=eap_1$. But then $eap_1q=ab$, which is impossible, since
from the rewriting system it follows that then $p_1$ and $q$ will
not contain $a$ and the normal form for $eap_1q$ must contain $e$.
Thus $p=fp_1$. Then $abp=abfp_1=xap_1$. Since $xap_1\to^{\ast} w'a$,
from the rewriting system it follows that $p_1$ starts either with
$f$ or with $b$ (note that $p_1$ cannot be the empty word as otherwise $p=f\in S\setminus T$):
\begin{itemize}
\item
$p_1=fp_2$. Then $abp=xafp_2=abp_2$.
\item
$p_1=bp_2$. Then $abp=xabp_2$.
\end{itemize}
By recursive arguments, we obtain that $w=x^k$ for some $k\geq 0$.
We have now $x^kaq\to^{\ast}ab$. Hence $q\in\{f,b,d\}^{\ast}$. In
actual fact $q$ cannot contain $d$'s. Indeed, otherwise we would need
to get rid of at least one of the $d$'s by applying the relations
$abd\to ea$ or $afd\to yea$, and in any case we would introduce $e$
to the left of $a$ which we would not be able to cancel. Hence
$q\in\{b,f\}^{\ast}$. Therefore in the derivation
$x^kaq\to^{\ast}ab$ the only relations we can apply are $f^3\to f$,
$abf\to xa$, $xaf\to ab$ and $af^2\to a$. It means that
$|x^kaq|_b+|x^kaq|_x=|ab|_b+|ab|_x$. Hence $k+|q|_b=1$. If $k=1$,
then $|q|_b=0$ an so $q=f$ or $q=f^2$. Since $q\in T$, we then
obtain that $q=f^2$ and so $ab=x^kaq=xaf^2=xa$, which is impossible.
Therefore $k=0$ and so $a\gR^T ab$. Hence we have the
derivation $abp\to^{\ast}a$. Again, in this derivation we cannot use
the rules $abd\to ea$ and $afd\to yea$. Since the only possibility
of obtain $y$ to the left of $a$ in the derivation, is to use the
rule $afd\to yea$, it follows that we cannot use in the derivation
the rule $yab\to af$. Hence the only rules we can use are $f^3\to
f$, $abf\to xa$, $xaf\to ab$ and $af^2\to a$. This yields $1\leq
|abp|_b+|abp|_x=|a|_b+|a|_x=0$, a contradiction. Thus
$(a,ab)\in\gJ^T\setminus\gD^T$ and so
$\gJ^T\neq\gD^T$.

Now we turn to proving that $\gJ^S=\gD^S$. To this
end let $u\gJ^S v$ for some $u,v\in S$. Then there exist
$\alpha,\beta,\gamma,\delta\in S$ such that $v=\alpha u\beta$ and
$u=\gamma v\delta$. Let the normal form for $u$ be
$u_1a^{k_1}u_2a^{k_2}\cdots u_sa^{k_s}u_{s+1}$, where $s\geq 0$ and
the words $u_i$ do not contain $a$. We have
\begin{equation}\label{eq:rewriting}
\gamma\alpha u_1a^{k_1}u_2a^{k_2}\cdots
u_sa^{k_s}u_{s+1}\beta\delta\to^{\ast}u_1a^{k_1}u_2a^{k_2}\cdots
u_sa^{k_s}u_{s+1}.
\end{equation}
First we aim to prove that
\begin{equation}\label{eq:achtung}
\gamma\alpha u_1a^{k_1}u_2a^{k_2}\cdots
u_sa^{k_s}u_{s+1}\beta\delta \gL^S \alpha
u_1a^{k_1}u_2a^{k_2}\cdots u_sa^{k_s}u_{s+1}\beta\delta.
\end{equation}
Obviously each of $\alpha$, $\beta$, $\gamma$ and $\delta$ does not
contain $a$. If $u_1$ contains $b$, $d$ or $f$, then from the
derivation~\eqref{eq:rewriting} and by inspection of the rules, we
obtain that $\gamma\alpha$ must be a power of $f$ and so in this
case~\eqref{eq:achtung} follows immediately. Hence we may assume
that $u_1\in\{c,e,x,y\}^{\ast}$. Analogously, if $\alpha$ contains
$b$, $d$ or $f$, then~\eqref{eq:rewriting} implies that $\gamma$
must be a power of $f$ and $\alpha$ must start with $f$; and in this
case again~\eqref{eq:achtung} holds. So, we may assume that
$\alpha\in\{c,e,x,y\}^{\ast}$ and then we are forced to have
$\gamma\in\{c,e,x,y\}^{\ast}$. If $s\neq 0$, then it immediately
implies~\eqref{eq:achtung}. If $s=0$, then none of the relations
involving $a$ can be applied and so we must have $\alpha=\gamma=1$,
again yielding ~\eqref{eq:achtung}. Thus~\eqref{eq:achtung} holds.
Now we will see when the condition
\begin{equation}\label{eq:gubaidulina}
\alpha u_1a^{k_1}u_2a^{k_2}\cdots
u_sa^{k_s}u_{s+1}\beta\delta \gR^S \alpha
u_1a^{k_1}u_2a^{k_2}\cdots u_sa^{k_s}u_{s+1}\beta
\end{equation}
holds and if it does, then this will finish the proof. From the
rewriting rules and~\eqref{eq:rewriting} it follows that each of
$\beta$ and $\delta$ cannot contain $e$, $c$, $x$ or $y$ and hence
$\delta\in\{b,d,f\}^{\ast}$. In the case when $s=0$, the
rewriting~\eqref{eq:rewriting} becomes $\gamma\alpha
u_1\beta\delta\to^{\ast}u_1$ in which no rules involving $a$ can be
applied and so we have that $\beta$ and $\delta$ can be only powers
of $f$; and then~\eqref{eq:gubaidulina} follows. Hence we may assume
that $s\neq 0$. Obviously we may also assume that $\delta\neq 1$. If
the derivation~\eqref{eq:rewriting} does not use a relation
involving the last $a$ in $a^{k_s}$, then $\beta$ and $\delta$ are
powers of $f$ and we are done. So, that \emph{marked} $a$ must be
involved in the derivation. We may assume that $\beta\delta$
contains $b$ or $d$ (otherwise $\beta\delta$ is a power of $f$ and
we are done). Assume that $|u_{s+1}|\geq 2$. Then $u_{s+1}$ cannot contain $d$'s and so
$u_{s+1}\in\{b,f\}^{\ast}$. Then $u_{s+1}$ cannot start with $f^2$
or $bf$ (otherwise in the normal form for $u$ we would have a
subword $af^2$ or $abf$). But if $u_{s+1}$ starts with $fb$ or
$b^2$, then $\delta$ is obliged to be a power of $f$ and $\beta$
must end with $f$, and then we are done. So, we may assume that
$|u_{s+1}|\leq 1$. Obviously $u_{s+1}\neq d$ and so we have three
possible cases: $u_{s+1}=1$, $u_{s+1}=b$ and $u_{s+1}=f$.

Assume that $s>1$ or $k_s>1$, then the relation involving the marked
$a$ cannot be $abd\to ea$ or $afd\to yea$, since otherwise we
introduce a new $e$, from which we will be not able to get rid off
in the derivation~\eqref{eq:rewriting}. From the
derivation~\eqref{eq:rewriting} it follows that
$au_{s+1}\beta\delta=\varepsilon au_{s+1}$ for some $\varepsilon\in
S$, and from the previous comment it follows that
$\varepsilon\in\{x,y\}^{\ast}$. Now we have three cases:
\begin{enumerate}

\item $u_{s+1}=1$. Then $a\beta\delta=\varepsilon a$.
Moreover, since $\varepsilon\in\{x,y\}^{\ast}$, and we may assume
that $\varepsilon a$ is in its normal form, $\varepsilon=x^k$ or
$\varepsilon=y^k$ for some $k\geq 0$. In the latter case, by the
rules $yab\to af$ and $af^2\to a$ we can find $\sigma\in S$ such
that $\varepsilon a\sigma=a$. Then $a\beta\delta\gR^S
a\beta$ and~\eqref{eq:gubaidulina} follows. Hence we may assume that
$\varepsilon=x^k$. Since $s>1$ or $k_s>1$, in the case when $k\neq
0$, it follows that in the derivation~\eqref{eq:rewriting} we either
will introduce several new $x$'s from which we cannot get rid off (in
the case when $k\geq 2$), or the normal form for $\gamma\alpha
u\beta\delta$ must end with $ab$ (in the case when $k=1$). Hence
$k=0$ and again $a\beta\delta\gR^S a\beta$,
yielding~\eqref{eq:gubaidulina}.

\item $u_{s+1}=b$. Then $ab\beta\delta=\varepsilon ab$.
Again $\varepsilon=x^k$ or $\varepsilon=y^k$ for some $k\geq 0$.
Since the rule $yab\to af$, we in actual fact have
$\varepsilon=x^k$. If $k\geq 1$, then in the
derivation~\eqref{eq:rewriting} we will introduce at least one new
$x$ from which we cannot get rid off. Hence $k=0$ and again
$ab\beta\delta\gR^Sab$, which
implies~\eqref{eq:gubaidulina}.

\item $u_{s+1}=f$. Then $af\beta\delta=\varepsilon af$
and again $\varepsilon=x^k$ or $\varepsilon=y^k$ for some $k\geq 0$.
Since the rule $xaf\to ab$, $\varepsilon=y^k$. Then applying the
rules $yab\to af$ and $af^2\to a$, we can find $\sigma\in S$ such
that $\varepsilon af\sigma=a$ and so $af\beta\delta\gR^Saf$.
Then~\eqref{eq:gubaidulina} holds.
\end{enumerate}

Thus we may assume that $s=1$ and $k_s=1$. To remind the situation: we
have $\gamma\alpha u_1au_2\beta\delta\to^{\ast} u_1au_2$ and we want
to prove that $u_1au_2\gD^S\alpha u_1au_2\beta$. Assume first that
$\gamma\alpha u_1$ contains $b$, $d$ or $f$. Then, as discussed
previously, in the derivation $\gamma\alpha
u_1au_2\beta\delta\to^{\ast} u_1au_2$ we cannot introduce $e$
immediately to the left of $a$, and so we cannot apply the rules
$abd\to ea$ and $afd\to yea$. Obviously we cannot also apply the rule
$yab\to af$. Then as in Cases 1--3, we obtain that
$au_2\beta\delta\gR^S au_2$ and so $\alpha
u_1au_2\beta\delta\gR^S\alpha u_1au_2\beta$. Thus we may assume that
$\gamma\alpha u_1\in\{e,c,x,y\}^{\ast}$. Then $\alpha
u_1au_2\beta \gL^S au_2\beta$. Moreover, we again have
$au_2\beta\delta=\varepsilon au_2$ for some
$\varepsilon\in\{e,c,x,y\}^{\ast}$. Then $u_1au_2=\gamma\alpha
u_1au_2\beta\delta=\gamma\alpha u_1\varepsilon au_2 \gL^S au_2$. So,
we are left to prove that $au_2\gD^S au_2\beta$. Recall that
$u_2\in\{1,b,f\}$ (as $|u_{s+1}|\leq 1$). Now, $a\gR^S af$ and
$a\gL^S xa\gR^S ab$, and so we are left to prove that
$au_2\beta\gD^S a$. But since $au_2\beta\delta\to^{\ast}\varepsilon
au_2$, we either have that $|u_2\beta|\leq 1$, or that
$au_2\beta=\varepsilon'a$ for some
$\varepsilon'\in\{e,c,x,y\}^{\ast}$. In the first case we have
$u_2\beta\in\{1,b,f\}$ and so $au_2\beta\gD^S a$; in the second case
$au_2\beta=\varepsilon' a\gL^S a$.

Thus we showed that $\gJ^S=\gD^S$ and we are done.
\end{proof}


\begin{thebibliography}{HKOT02}

\bibitem[Adj66]{adjan_defining}
S.~I. Adjan.
\newblock `Defining relations and algorithmic problems for groups and 		  semigroups'.
\newblock {\em Proceedings of the Steklov Institute of Mathematics}, 85 (1966) .
\newblock [Translated from the Russian by M.~Greendlinger.].

\bibitem[ARS]{araujo_presentations}
I. Ara\'{u}jo, N. Ru\v{s}kuc, \& P.~V. Silva.
\newblock `Presentations for inverse subsemigroups with finite 		  complement'.
\newblock Preprint.

\bibitem[Ber06]{bergman_generating}
G.~M. Bergman.
\newblock `Generating infinite symmetric groups'.
\newblock {\em Bull. London Math. Soc.}, 38, no.~3 (2006), pp. 429--440.
\newblock {\sc doi:} \href {http://dx.doi.org/10.1112/S0024609305018308} {{10.1112/S0024609305018308}}.

\bibitem[Blu99]{blumensath_diploma}
A. Blumensath.
\newblock {\em Automatic Structures}.
\newblock Diploma thesis, RWTH Aachen, 1999.
\newblock {\sc url:} \href{http://www.mathematik.tu-darmstadt.de/~blumensath/Publications/AutStr.pdf}{\nolinkurl{www.mathematik.tu-darmstadt.de/~blumensath/Publications/AutStr.pdf}}.

\bibitem[BO93]{book_srs}
R.~V. Book \& F. Otto.
\newblock {\em {String-Rewriting Systems}}.
\newblock Texts and Monographs in Computer Science. Springer-Verlag, New York, 1993.

\bibitem[Cai05]{c_phdthesis}
A.~J. Cain.
\newblock {\em Presentations for Subsemigroups of Groups}.
\newblock {Ph.D.\ Thesis}, University of St~Andrews, 2005.
\newblock {\sc url:} \href{http://www-groups.mcs.st-andrews.ac.uk/~alanc/pub/c_phdthesis.pdf}{\nolinkurl{www-groups.mcs.st-andrews.ac.uk/~alanc/pub/c_phdthesis.pdf}}.

\bibitem[Cai07]{c_malcev}
A.~J. Cain.
\newblock `Malcev presentations for subsemigroups of groups --- a 		  survey'.
\newblock In C.~M. Campbell, M. Quick, E.~F. Robertson, \& G.~C. Smith, eds, {\em Groups St Andrews 2005 {\rm (Vol. 1)}}, no. 339 in {\em London Mathematical Society Lecture Note Series}, pp. 256--268, Cambridge, 2007. Cambridge University Press.

\bibitem[CGR12]{cgr_greenindex}
A.~J. Cain, R. Gray, \& N. Ru{\v{s}}kuc.
\newblock `Green index in semigroup theory: generators, 		  presentations, and automatic structures'.
\newblock {\em Semigroup Forum}, 85, no.~3 (2012), pp. 448--476.
\newblock {\sc doi:} \href {http://dx.doi.org/10.1007/s00233-012-9406-2} {{10.1007/s00233-012-9406-2}}.

\bibitem[CMa]{cm_hopf}
A.~J. Cain \& V. Maltcev.
\newblock `Hopfian and co-hopfian subsemigroups and extensions'.
\newblock arXiv:~\href {http://arxiv.org/abs/1305.6176} {{1305.6176}}.

\bibitem[CMb]{cm_markov}
A.~J. Cain \& V. Maltcev.
\newblock `Markov semigroups, monoids, and groups'.
\newblock arXiv:~\href {http://arxiv.org/abs/1202.3013} {{1202.3013}}.

\bibitem[CM12]{cm_wordhypunique}
A.~J. Cain \& V. Maltcev.
\newblock `Context-free rewriting systems and word-hyperbolic 		  structures with uniqueness'.
\newblock {\em Internat. J. Algebra Comput.}, 22, no.~7 (2012) .
\newblock {\sc doi:} \href {http://dx.doi.org/10.1142/S0218196712500610} {{10.1142/S0218196712500610}}.

\bibitem[CORT08]{cort_apcancsg}
A.~J. Cain, G. Oliver, N. Ru{\v{s}}kuc, \& R.~M. Thomas.
\newblock `Automatic presentations for cancellative semigroups'.
\newblock In C. Mart\'{\i}n-Vide, H. Fernau, \& F. Otto, eds, {\em Language and Automata Theory and Applications: Second 		  International Conference, Tarragona, Spain, March 13--19, 		  2008}, no. 5196 in {\em Lecture Notes in Computer Science}, pp. 149--159. Springer, 2008.
\newblock {\sc doi:} \href {http://dx.doi.org/10.1007/978-3-540-88282-4\_15} {{10.1007/978-3-540-88282-4\_15}}.

\bibitem[CORT09]{cort_apsg}
A.~J. Cain, G. Oliver, N. Ru{\v{s}}kuc, \& R.~M. Thomas.
\newblock `Automatic presentations for semigroups'.
\newblock {\em Inform. and Comput.}, 207, no.~11 (2009), pp. 1156--1168.
\newblock {\sc doi:} \href {http://dx.doi.org/10.1016/j.ic.2009.02.005} {{10.1016/j.ic.2009.02.005}}.

\bibitem[CORT10]{cort_const}
A.~J. Cain, G. Oliver, N. Ru{\v{s}}kuc, \& R.~M. Thomas.
\newblock `Automatic presentations and semigroup constructions'.
\newblock {\em Theory Comput. Syst.}, 47, no.~2 (2010), pp. 568--592.
\newblock {\sc doi:} \href {http://dx.doi.org/10.1007/s00224-009-9216-4} {{10.1007/s00224-009-9216-4}}.

\bibitem[CP61]{clifford_semigroups1}
A.~H. Clifford \& G.~B. Preston.
\newblock {\em {The Algebraic Theory of Semigroups {\rm (Vol.~I)}}}.
\newblock No.~7 in {\em Mathematical Surveys}. American Mathematical Society, Providence, R.I., 1961.

\bibitem[CP67]{clifford_semigroups2}
A.~H. Clifford \& G.~B. Preston.
\newblock {\em {The Algebraic Theory of Semigroups {\rm (Vol.~II)}}}.
\newblock No.~7 in {\em Mathematical Surveys}. American Mathematical Society, Providence, R.I., 1967.

\bibitem[Cro54]{croisot_automorphisms}
R. Croisot.
\newblock `Automorphismes int\'erieurs d'un semi-groupe'.
\newblock {\em Bull. Soc. Math. France}, 82 (1954), pp. 161--194.
\newblock [In French].

\bibitem[CRR08]{crr_finind}
A.~J. Cain, E.~F. Robertson, \& N. Ru{\v{s}}kuc.
\newblock `Cancellative and {M}alcev presentations for finite {R}ees 		  index subsemigroups and extensions'.
\newblock {\em J. Aust. Math. Soc.}, 84, no.~1 (2008), pp. 39--61.
\newblock {\sc doi:} \href {http://dx.doi.org/10.1017/s1446788708000086} {{10.1017/s1446788708000086}}.

\bibitem[CRRT95]{campbell_reidemeister}
C.~M. Campbell, E.~F. Robertson, N. Ru{\v{s}}kuc, \& R.~M. Thomas.
\newblock `Reidemeister--{S}chreier type rewriting for semigroups'.
\newblock {\em Semigroup Forum}, 51, no.~1 (1995), pp. 47--62.

\bibitem[CRRT01]{campbell_autsg}
C.~M. Campbell, E.~F. Robertson, N. Ru{\v{s}}kuc, \& R.~M. Thomas.
\newblock `Automatic semigroups'.
\newblock {\em Theoret. Comput. Sci.}, 250, no.~1--2 (2001), pp. 365--391.
\newblock {\sc doi:} \href {http://dx.doi.org/10.1016/S0304-3975(99)00151-6} {{10.1016/S0304-3975(99)00151-6}}.

\bibitem[CRT12]{crt_unaryfa}
A.~J. Cain, N. Ru{\v{s}}kuc, \& R.~M. Thomas.
\newblock `Unary {FA}-presentable semigroups'.
\newblock {\em Internat. J. Algebra Comput.}, 22, no.~4 (2012) .
\newblock {\sc doi:} \href {http://dx.doi.org/10.1142/S0218196712500385} {{10.1142/S0218196712500385}}.

\bibitem[DG04]{duncan_hyperbolic}
A. Duncan \& R.~H. Gilman.
\newblock `Word hyperbolic semigroups'.
\newblock {\em Math. Proc. Cambridge Philos. Soc.}, 136, no.~3 (2004), pp. 513--524.
\newblock {\sc doi:} \href {http://dx.doi.org/10.1017/S0305004103007497} {{10.1017/S0305004103007497}}.

\bibitem[ECH{\etalchar{+}}92]{epstein_wordproc}
D.~B. A. Epstein, J.~W. Cannon, D.~F. Holt, S.~V. F. Levy, M.~S. Paterson, \& W.~P. Thurston.
\newblock {\em {Word Processing in Groups}}.
\newblock Jones \& Bartlett, Boston, Mass., 1992.

\bibitem[Edw83]{edwards_eventually}
P.~M. Edwards.
\newblock `Eventually regular semigroups'.
\newblock {\em Bull. Austral. Math. Soc.}, 28, no.~1 (1983), pp. 23--38.
\newblock {\sc doi:} \href {http://dx.doi.org/10.1017/S0004972700026095} {{10.1017/S0004972700026095}}.

\bibitem[GdlH90]{ghys_hyperbolic}
{\'E}. Ghys \& P. de~la Harpe, eds.
\newblock {\em Sur les groupes hyperboliques d'apr\`es {M}ikhael 		  {G}romov}, vol.~83 of {\em Progress in Mathematics}.
\newblock Birkh\"auser Boston Inc., Boston, MA, 1990.
\newblock Papers from the Swiss Seminar on Hyperbolic Groups held in 		  Bern, 1988. [Partial English translation by W. Grosso: 		  \url{www.umpa.ens-lyon.fr/~ghys/articles/Ghys-delaHarpe-english.pdf}].

\bibitem[Gil02]{gilman_hyperbolic}
R.~H. Gilman.
\newblock `On the definition of word hyperbolic groups'.
\newblock {\em Math. Z.}, 242, no.~3 (2002), pp. 529--541.
\newblock {\sc doi:} \href {http://dx.doi.org/10.1007/s002090100356} {{10.1007/s002090100356}}.

\bibitem[GMMR]{gray_ideals}
R. Gray, V. Maltcev, J.~D. Mitchell, \& N. Ru\v{s}kuc.
\newblock `Ideals, finiteness conditions and green index for 		  subsemigroups'.
\newblock Preprint.
\newblock arXiv:~\href {http://arxiv.org/abs/1204.6602} {{1204.6602}}.

\bibitem[GP98]{guba_cohomological}
V.~S. Guba \& S.~J. Pride.
\newblock `On the left and right cohomological dimension of monoids'.
\newblock {\em Bull. London Math. Soc.}, 30, no.~4 (1998), pp. 391--396.
\newblock {\sc doi:} \href {http://dx.doi.org/10.1112/S0024609398004676} {{10.1112/S0024609398004676}}.

\bibitem[GP11]{gray_homological}
R. Gray \& S.~J. Pride.
\newblock `Homological finiteness properties of monoids, their ideals 		  and maximal subgroups'.
\newblock {\em J. Pure Appl. Algebra}, 215, no.~12 (2011), pp. 3005--3024.
\newblock {\sc doi:} \href {http://dx.doi.org/10.1016/j.jpaa.2011.04.019} {{10.1016/j.jpaa.2011.04.019}}.

\bibitem[GR08]{gray_green1}
R. Gray \& N. Ru{\v{s}}kuc.
\newblock `Green index and finiteness conditions for semigroups'.
\newblock {\em J. Algebra}, 320, no.~8 (2008), pp. 3145--3164.
\newblock {\sc doi:} \href {http://dx.doi.org/10.1016/j.jalgebra.2008.07.008} {{10.1016/j.jalgebra.2008.07.008}}.

\bibitem[GR11]{gray_boundaries}
R. Gray \& N. Ru{\v{s}}kuc.
\newblock `Generators and relations for subsemigroups via boundaries 		  in {C}ayley graphs'.
\newblock {\em J. Pure Appl. Algebra}, 215, no.~11 (2011), pp. 2761--2779.
\newblock {\sc doi:} \href {http://dx.doi.org/10.1016/j.jpaa.2011.03.017} {{10.1016/j.jpaa.2011.03.017}}.

\bibitem[Gri88]{grigorchuk_semigroups}
R.~I. Grigorchuk.
\newblock `Semigroups with cancellations of degree growth'.
\newblock {\em Mat. Zametki}, 43, no.~3 (1988), pp. 305--319, 428.
\newblock {\sc doi:} \href {http://dx.doi.org/10.1007/BF01138837} {{10.1007/BF01138837}}.

\bibitem[Gro87]{gromov_hyperbolic}
M. Gromov.
\newblock `Hyperbolic groups'.
\newblock In S.~M. Gersten, ed., {\em Essays in group theory}, vol.~8 of {\em Math. Sci. Res. Inst. Publ.}, pp. 75--263. Springer, New York, 1987.

\bibitem[HKOT02]{hoffmann_relatives}
M. Hoffmann, D. Kuske, F. Otto, \& R.~M. Thomas.
\newblock `Some relatives of automatic and hyperbolic groups'.
\newblock In G.~M. S. Gomes, J.~{\'{E}}. Pin, \& P.~V. Silva, eds, {\em Semigroups, Algorithms, Automata and Languages ({C}oimbra, 		  2001)}, pp. 379--406. World Scientific Publishing, River Edge, N.J., 2002.

\bibitem[How95]{howie_fundamentals}
J.~M. Howie.
\newblock {\em Fundamentals of Semigroup Theory}, vol.~12 of {\em London Mathematical Society Monographs {\rm(New Series)}}.
\newblock Clarendon Press, Oxford University Press, New York, 1995.

\bibitem[HTR02]{hoffmann_autofinrees}
M. Hoffmann, R.~M. Thomas, \& N. Ru{\v{s}}kuc.
\newblock `Automatic semigroups with subsemigroups of finite {R}ees 		  index'.
\newblock {\em Internat. J. Algebra Comput.}, 12, no.~3 (2002), pp. 463--476.
\newblock {\sc doi:} \href {http://dx.doi.org/10.1142/S0218196702000833} {{10.1142/S0218196702000833}}.

\bibitem[Jur78]{jura_ideals}
A. Jura.
\newblock `Determining ideals of a given finite index in a finitely 		  presented semigroup'.
\newblock {\em Demonstratio Math.}, 11, no.~3 (1978), pp. 813--827.

\bibitem[KMC]{kilibarda_ends}
V. Kilibarda, V. Maltcev, \& S. Craik.
\newblock `Ends for subsemigroups of finite index'.
\newblock Preprint.
\newblock arXiv:~\href {http://arxiv.org/abs/http://arxiv.org/abs/1302.3500} {{http://arxiv.org/abs/1302.3500}}.

\bibitem[KN95]{khoussainov_autopres}
B. Khoussainov \& A. Nerode.
\newblock `Automatic presentations of structures'.
\newblock In {\em Logic and computational complexity (Indianapolis, IN, 		  1994)}, vol. 960 of {\em Lecture Notes in Computer Science}, pp. 367--392. Springer, Berlin, 1995.
\newblock {\sc doi:} \href {http://dx.doi.org/10.1007/3-540-60178-3\_93} {{10.1007/3-540-60178-3\_93}}.

\bibitem[Kob10]{kobayashi_homological}
Y. Kobayashi.
\newblock `The homological finiteness properties left-, right-, and 		  bi-{${\rm FP}_n$} of monoids'.
\newblock {\em Comm. Algebra}, 38, no.~11 (2010), pp. 3975--3986.
\newblock {\sc doi:} \href {http://dx.doi.org/10.1080/00927872.2010.507562} {{10.1080/00927872.2010.507562}}.

\bibitem[Lal79]{lallement_semigroups}
G. Lallement.
\newblock {\em Semigroups and Combinatorial Applications}.
\newblock John Wiley \& Sons, New York-Chichester-Brisbane, 1979.

\bibitem[Mal39]{malcev_immersion1}
A.~I. Malcev.
\newblock `On the immersion of associative systems in groups'.
\newblock {\em Mat.\ Sbornik}, 6, no.~48 (1939), pp. 331--336.
\newblock [In Russian.].

\bibitem[Mal07]{malheiro_trivializers}
A. Malheiro.
\newblock `On trivializers and subsemigroups'.
\newblock In {\em Semigroups and formal languages}, pp. 188--204. World Sci. Publ., Hackensack, NJ, 2007.
\newblock {\sc doi:} \href {http://dx.doi.org/10.1142/9789812708700\_0013} {{10.1142/9789812708700\_0013}}.

\bibitem[Mal09]{malheiro_fdt}
A. Malheiro.
\newblock `Finite derivation type for large ideals'.
\newblock {\em Semigroup Forum}, 78, no.~3 (2009), pp. 450--485.
\newblock {\sc doi:} \href {http://dx.doi.org/10.1007/s00233-008-9109-x} {{10.1007/s00233-008-9109-x}}.

\bibitem[MMR09]{maltcev_bergman}
V. Maltcev, J.~D. Mitchell, \& N. Ru{\v{s}}kuc.
\newblock `The {B}ergman property for semigroups'.
\newblock {\em J. Lond. Math. Soc. (2)}, 80, no.~1 (2009), pp. 212--232.
\newblock {\sc doi:} \href {http://dx.doi.org/10.1112/jlms/jdp025} {{10.1112/jlms/jdp025}}.

\bibitem[MR]{maltcev_hopfian}
V. Maltcev \& N. Ru\v{s}kuc.
\newblock `On hopfian cofinite subsemigroups'.
\newblock Submitted.

\bibitem[NT08]{nies_rings}
A. Nies \& R.~M. Thomas.
\newblock `F{A}-presentable groups and rings'.
\newblock {\em J. Algebra}, 320, no.~2 (2008), pp. 569--585.
\newblock {\sc doi:} \href {http://dx.doi.org/10.1016/j.jalgebra.2007.04.015} {{10.1016/j.jalgebra.2007.04.015}}.

\bibitem[OT05]{oliver_autopresgroups}
G.~P. Oliver \& R.~M. Thomas.
\newblock `Automatic presentations for finitely generated groups'.
\newblock In V. Diekert \& B. Durand, eds, {\em 22nd Annual Symposium on Theoretical Aspects of Computer 		  Science (STACS'05), Stuttgart, Germany}, vol. 3404 of {\em Lecture Notes in Comput. Sci.}, pp. 693--704, Berlin, 2005. Springer.
\newblock {\sc doi:} \href {http://dx.doi.org/10.1007/978-3-540-31856-9\_57} {{10.1007/978-3-540-31856-9\_57}}.

\bibitem[Rip82]{rips_subgroups}
E. Rips.
\newblock `Subgroups of small cancellation groups'.
\newblock {\em Bull. London Math. Soc.}, 14, no.~1 (1982), pp. 45--47.
\newblock {\sc doi:} \href {http://dx.doi.org/10.1112/blms/14.1.45} {{10.1112/blms/14.1.45}}.

\bibitem[RT98]{ruskuc_syntactic}
N. Ru{\v{s}}kuc \& R.~M. Thomas.
\newblock `Syntactic and {R}ees indices of subsemigroups'.
\newblock {\em J. Algebra}, 205, no.~2 (1998), pp. 435--450.
\newblock {\sc doi:} \href {http://dx.doi.org/10.1006/jabr.1997.7392} {{10.1006/jabr.1997.7392}}.

\bibitem[Ru{\v{s}}98]{ruskuc_largesubsemigroups}
N. Ru{\v{s}}kuc.
\newblock `On large subsemigroups and finiteness conditions of 		  semigroups'.
\newblock {\em Proc. London Math. Soc. (3)}, 76, no.~2 (1998), pp. 383--405.
\newblock {\sc doi:} \href {http://dx.doi.org/10.1112/S0024611598000124|} {{10.1112/S0024611598000124|}}.

\bibitem[SOK94]{squier_finiteness}
C.~C. Squier, F. Otto, \& Y. Kobayashi.
\newblock `A finiteness condition for rewriting systems'.
\newblock {\em Theoret. Comput. Sci.}, 131, no.~2 (1994), pp. 271--294.
\newblock {\sc doi:} \href {http://dx.doi.org/10.1016/0304-3975(94)90175-9} {{10.1016/0304-3975(94)90175-9}}.

\bibitem[Spe77]{spehner_presentations}
J.~C. Spehner.
\newblock `Pr\'esentations et pr\'esentations simplifiables d'un 		  mono{\"{\i}}de simplifiable'.
\newblock {\em Semigroup Forum}, 14, no.~4 (1977), pp. 295--329.
\newblock [In French.].
\newblock {\sc doi:} \href {http://dx.doi.org/10.1007/BF02194675} {{10.1007/BF02194675}}.

\bibitem[Wan98]{wang_fcrsfdt}
J. Wang.
\newblock `Finite complete rewriting systems and finite derivation 		  type for small extensions of monoids'.
\newblock {\em J. Algebra}, 204, no.~2 (1998), pp. 493--503.
\newblock {\sc doi:} \href {http://dx.doi.org/10.1006/jabr.1997.7388} {{10.1006/jabr.1997.7388}}.

\bibitem[WW11]{wong_fcrs}
K. Wong \& P. Wong.
\newblock `On finite complete rewriting systems and large 		  subsemigroups'.
\newblock {\em J. Algebra}, 345, no.~1 (2011), pp. 242--256.
\newblock arXiv:~\href {http://arxiv.org/abs/1005.0882v2} {{1005.0882v2}}, {\sc doi:} \href {http://dx.doi.org/10.1016/j.jalgebra.2011.08.022} {{10.1016/j.jalgebra.2011.08.022}}.

\end{thebibliography}

\newcommand{\etalchar}[1]{$^{#1}$}

\end{document}